%% file: main.tex
\documentclass[10pt]{article}

\usepackage[utf8]{inputenc}
\usepackage[T1]{fontenc}

\usepackage[top=3cm, bottom=3cm, left=3cm, right=3cm]{geometry}

\usepackage{amsmath,amssymb,amsfonts,amsthm}
\usepackage{algorithmic}
\usepackage{graphicx}
\usepackage{textcomp}
\usepackage{xcolor}
\usepackage{float}

\usepackage{tikz}
\usepackage{pgf,tikz}
\usepackage{graphicx}
\usepackage{enumitem}
\usepackage[noadjust]{cite}
\usepackage{dsfont}
\usepackage{amsmath}
\usepackage{tabularx}
\usepackage{booktabs}
\usepackage{caption}
\usepackage{array}

\usepackage[normalem]{ulem}

\usepackage{url}

\usepackage{lmodern}
    
\theoremstyle{plain}
\newtheorem{theorem}{Theorem}  
\newtheorem{proposition}{Proposition}
\newtheorem{corollary}{Corollary}
\newtheorem{lemma}{Lemma}
\theoremstyle{definition}
\newtheorem{remark}{Remark}

\newcommand{\R}{\mathbb{R}}
\newcommand{\N}{\mathbb{N}}
\newcommand{\C}{\mathbb{C}}
\newcommand{\diff}{\mathrm{d}}

\DeclareMathOperator{\Ln}{Ln}
\DeclareMathOperator{\real}{Re}
\DeclareMathOperator{\imag}{Im}
\DeclareMathOperator{\sign}{sgn}

\begin{document}

\title{A novel necessary and sufficient condition for the stability of $2\times 2$ first-order linear hyperbolic systems}

\author{Ismaïla Balogoun\thanks{Université Paris-Saclay, CNRS, CentraleSupélec, Inria, Laboratoire des Signaux et Systèmes, 91190 Gif-sur-Yvette, France} \and Jean Auriol\footnotemark[1] \and Islam Boussaada\footnotemark[1] \thanks{IPSA Paris, 94200 Ivry-sur-Seine, France} \and Guilherme Mazanti\footnotemark[1] \thanks{Fédération de Mathématiques de CentraleSupélec, 91190, Gif-sur-Yvette, France}}

\date{}

\maketitle

\begin{abstract}
In this paper, we establish a necessary and sufficient stability condition for a class of two coupled first-order linear hyperbolic partial differential equations. Through a backstepping transform, the problem is reformulated as a stability problem for an integral difference equation, that is, a difference equation with distributed delay. Building upon a Stépán--Hassard argument variation theorem originally designed for time-delay systems of retarded type, we then introduce a theorem that counts the number of unstable roots of our integral difference equation. This leads to the expected necessary and sufficient stability criterion for the system of first-order linear hyperbolic partial differential equations. Finally, we validate our theoretical findings through simulations.

\bigskip

\end{abstract}

\section{Introduction} 
Systems of first-order hyperbolic partial differential equations (PDEs) have been extensively studied over the years due to their application in modeling various physical phenomena. These include drilling devices \cite{auriol2020closed,auriol2022comparing}, water management systems \cite{diagne2017control}, aeronomy \cite{schunk1975transport}, cable vibration dynamics \cite{wang2020vibration} and pipelines \cite{rager2015simplified}, traffic networks~\cite{espitia2022traffic}. Comprehensive overviews of current research in this field can be found in  \cite{bastin2016stability} and \cite{hayat2021boundary}.

When there are no in-domain coupling terms, the method of characteristics can be used to relate the behavior of these systems to time-delay systems, a link extensively explored in the literature \cite{Chitour2024Approximate, Chitour2021One, Chitour2016Stability, cooke1968differential, slemrod1971nonexistence}. The situation becomes more complicated when in-domain coupling terms are present. To overcome this difficulty, a possible strategy, adopted for instance in \cite{saba2019stability, auriol2019explicit, auriol2019sensing}, is to make use of the backstepping transform for hyperbolic systems from \cite{krstic2008boundary} to convert these coupling terms into integral boundary terms, and then use the method of characteristics to relate the behavior of these systems to integral difference equations.

Control questions concerning time-delay systems present intriguing and complex mathematical challenges. For linear time-invariant systems, stability and stabilization issues can be addressed through spectral methods, as detailed in \cite{hale1993introduction, michiels2014stability}. Similarly to linear time-invariant finite-dimensional systems, the stability of time-delay systems can be characterized through the position of its spectrum with respect to the imaginary axis. For instance, exponential stability is equivalent to the existence of $\alpha>0$ such that $\real s \leq-\alpha$ for every $s$ in the spectrum of the time-delay system.

The spectrum of systems with finitely many discrete delays is made only of eigenvalues, which are infinite in number and can be characterized as the complex roots of a quasipolynomial. Quasipolynomials were extensively studied in the literature, and characterizing the location of their roots is a challenging problem that has attracted much research effort, both from theoretical and numerical points of view \cite{avellar1980zeros, berenstein2012complex, vyhlidal2014qpmr, stepan1989retarded, hassard1997counting}. Important insights on the stability of difference equations with finitely many discrete delays, including a discussion of robustness with respect to the delays, can be found in \cite[Chapter~9, Section~6]{hale1993introduction}, with extensions to systems with time-varying parameters provided in \cite{Chitour2016Stability}. One can also obtain necessary or sufficient stability conditions using Lyapunov--Krasovskii functionals, as done in \cite{campos2018necessary} for systems with finitely many discrete delays and in \cite{ortiz2022necessary} for integral difference equations.

Cauchy's argument principle, which is a standard result in complex analysis, turns out to be an efficient way to investigate the stability of delay systems. It is used, in particular, in the St\'{e}p\'{a}n--Hassard argument variation approach, which allows one to count the number of eigenvalues with positive real parts. The corresponding counting formula, which was first introduced in \cite{stepan1989retarded} and then refined in \cite{hassard1997counting}, remains relevant not only for stability purposes but also for recent developments in the stabilization of time-delay systems, through the so-called \emph{partial pole placement} method. Indeed, this method consists in selecting the system parameters in order to enforce a finite number of prescribed eigenvalues of the system, and the dominance of these assigned eigenvalues is often shown by exploiting the St\'{e}p\'{a}n--Hassard counting formula, with stabilization being achieved when all the chosen eigenvalues have negative real parts \cite{boussaada2016characterizing, bedouhene2020real, Fueyo2023Pole, Boussaada2022Generic}. While the partial pole placement method has shown its effectiveness in the prescribed stabilization of scalar first-order hyperbolic partial differential equations as discussed in \cite{Boussaada2022Generic,ammari:hal-04200203,benarab:hal-04196450, schmoderer:hal-04194365}, the application of the underlying Hille oscillation Theorem  \cite{hille1922} appears to be computationally cumbersome when dealing with systems of coupled hyperbolic partial differential equations.  

In this paper, we analyze the stability of a class of $2\times 2$ linear hyperbolic coupled PDEs. We use the backstepping transformation from \cite{auriol2018delay} to transform the original system into a target system that can be written as an integral difference equation. Then, we address the stability of the integral difference equation through spectral methods. Although the equation is not of the retarded type, we show that St\'{e}p\'{a}n--Hassard arguments from \cite{stepan1989retarded, hassard1997counting} can be adapted to count the number of roots with strictly positive real parts, assuming there are no roots on the imaginary axis. Combined with an analysis of the vertical asymptotes of the roots of the characteristic function, we obtain, as a consequence, necessary and sufficient conditions for the stability of the difference equation in question, which will yield the same results for the original hyperbolic system.

This paper is organized as follows. Section~\ref{sec_description} presents the system under consideration and recalls the results of \cite{saba2019stability} transforming our system into an integral difference equation. Section~\ref{sec_main} contains our main results: Theorem~\ref{thm_open} uses the St\'{e}p\'{a}n--Hassard approach to count the number of unstable roots of our integral difference equation, Corollary~\ref{coro:stability} obtains as a consequence necessary and sufficient conditions for exponential stability, and Corollary~\ref{result:eqhyp} uses the relation between the integral difference equation and our original system in order to obtain necessary and sufficient conditions for exponential stability of the latter. A comparison between our results and some other results available in the literature is provided in Section~\ref{sec_comparaison}, while, in Section~\ref{sec_approximation}, we focus on the case where the in-domain coupling terms are constant and provide some numerical insights based on polynomial approximations of the integral kernel of the integral difference equation. The paper is concluded by final remarks and perspectives provided in Section~\ref{sec_conclusion}.

\paragraph*{Notation} In this paper, the principal branch of the complex logarithm is denoted by $\Ln$, and the principal argument of a complex number is denoted $\arg$. Given a measurable subset $\Omega$ of $\mathbb{R}$ with positive measure, $L^2(\Omega,\R)$ denotes the set of (Lebesgue) measurable functions $f$ mapping the set $\Omega$ into $\mathbb{R}$ such that $\int_\Omega \lvert f(x) \rvert^2 \diff x < +\infty$, identifying functions which are equal almost everywhere. The associated norm is $\lVert f\rVert_{L^2(\Omega)}^2:= \int_\Omega \lvert f(x)\rvert^2 \diff x$. The Sobolev space $H^{1}(\Omega,\R)$ is defined as the set $\lbrace f \in L^2(\Omega,\R)\mid f^\prime\in L^2(\Omega,\R)\rbrace$, where the derivative is to be understood in the sense of distributions. Given a delay $\tau > 0$, a function $z\colon [-\tau, \infty) \mapsto \mathbb{R}$, and $t \geq 0$, the history function of $z$ at time $t$ is the function $z_{[t]} \colon [- \tau, 0] \to \mathbb R$ defined by $z_{[t]}(\theta) = z(t + \theta)$ for $\theta \in [-\tau, 0]$. For a set $I$, $\mathds{1}_I(x)$ is the function defined by
\[
\mathds{1}_I(x)
=
\begin{cases}
1 & \text { if } x \in I, \\
0 & \text { otherwise}.
\end{cases}
\]

\section{Problem description}
\label{sec_description}

\subsection{System under consideration}
We are interested in the stability analysis of the linear hyperbolic system
\begin{equation} \label{eq:hyperbolic_couple}
\left\{
\begin{aligned}
&u_t(t, x)+\lambda u_x(t, x)=\sigma^{+}(x) v(t, x),~\,t>0,~ x \in[0,1], \\
&v_t(t, x)-\mu v_x(t, x)=\sigma^{-}(x) u(t, x),\,~t>0,~ x \in[0,1],\\
&u(t, 0)=q v(t, 0),\,~t>0,\\
&v(t, 1)=\rho u(t, 1),\,~t>0,
\end{aligned}
\right.
\end{equation}
where $(u(t,x), v(t,x))^T$ is the state of the system, the different arguments evolving in $\{(t,x) \mid t>0,~ x \in [0,1] \}$. The in-domain coupling terms $\sigma^{+}$ and $\sigma^{-}$ are assumed to be continuous functions, whereas the boundary coupling terms $q$ and $\rho$ and the velocities $\lambda>0$ and $\mu>0$ are assumed to be constant. We denote by $u_0(\cdot)=u(0,\cdot)$ and $v_0(\cdot)=v(0,\cdot)$ the initial conditions associated with~\eqref{eq:hyperbolic_couple}. We assume here that they belong to $H^1((0,1),\mathbb{R})$ and satisfy the compatibility conditions 
\begin{align}
u_0(0)=qv_0(0),\quad v_0(1)=\rho u_0(1). \label{compatibility_condition_u_v}
\end{align}
As shown in \cite[Appendix~A]{bastin2016stability}, system~\eqref{eq:hyperbolic_couple} with initial condition $(u_0,v_0)$ in $H^1((0, 1), \mathbb R^2)$ satisfying the compatibility condition~\eqref{compatibility_condition_u_v} is well-posed, and its solution $(u, v)$ belongs to the space $C^1([0, +\infty),\allowbreak L^2((0, 1), \mathbb R^2)) \cap C^0([0, +\infty),\allowbreak H^1((0, 1), \mathbb R^2))$.
In the sequel, we define the characteristic time $\tau$ of the system as
\begin{equation}\label{tau}
\tau=\frac{1}{\lambda}+\frac{1}{\mu}.
\end{equation}
Finally, we assume $q\neq 0$. Although the computations can be adjusted to deal with the case $q = 0$, we make this simplifying assumption here for the sake of clarity of presentation, and we direct the reader to \cite[Section~3.5]{coron2013local} for the scenario where $q=0$.

\subsection{Objective and methodology}

Our objective is to construct necessary and sufficient stability conditions that guarantee the exponential stability of system~\eqref{eq:hyperbolic_couple} in $L^2$ norm, namely, the existence of $\nu>0$ and $C\geq 0$ such that, for any $(u_0,v_0) \in H^1([0,1],\mathbb{R})\times H^1([0,1],\mathbb{R})$ satisfying the compatibility condition~\eqref{compatibility_condition_u_v}, the solution $(u,v)$ of system~\eqref{eq:hyperbolic_couple}  satisfies
\begin{align}
\lVert(u(t,\cdot),v(t,\cdot))\rVert_{(L^2(0,1))^2} \leq C\mathrm{e}^{-\nu t} \lVert (u_0,v_0) \rVert_{(L^2(0,1))^2},~t\geq 0. 
\end{align}

As explained in the introduction, stability conditions for systems of conservation and balance laws can be found in the literature~\cite{bastin2016stability}. Most of the existing results are based on (weighted $L^2$) Lyapunov functions and linear matrix inequalities (LMIs) and are, therefore, sufficient only conditions. It has been shown in~\cite{auriol2019explicit} that systems of first-order hyperbolic PDEs share equivalent stability properties to those of a class of integral difference equations (IDEs). This representation has been successfully used in~\cite{saba2019stability} to obtain a new stability condition. However, the proposed condition could not be easily verified and had to be relaxed to a sufficient-only condition to be implemented. In this paper, we use the same time-delay system framework to characterize the unstable roots of the system~\eqref{eq:hyperbolic_couple} and obtain \emph{implementable} necessary and sufficient stability conditions.

\subsection{Equivalent integral difference equation}

In this section, we adopt the approach presented in~\cite{saba2019stability, auriol2019explicit} to rewrite the PDE system~\eqref{eq:hyperbolic_couple} as an IDE with equivalent stability properties. To do so, we use a classical backstepping transformation. The detailed computations can be found in~\cite{saba2019stability} and we only recall here the main results that will be of use to us in the sequel.

Let us consider the Volterra change of coordinates defined in \cite{coron2013local}, given by
\begin{equation}\label{backstepping}
\begin{aligned}
\alpha(t, x) & =u(t, x)-\int_0^x\left(K^{u u}(x, \xi) u(t, \xi)+K^{u v}(x, \xi) v(t, \xi)\right) \diff \xi, \\
\beta(t, x) & =v(t, x)-\int_0^x\left(K^{v u}(x, \xi) u(t, \xi)+K^{v v}(x, \xi) v(t, \xi)\right) \diff \xi,
\end{aligned}
\end{equation}
where the kernels $K^{u u}, K^{u v}, K^{v u}, K^{v v}$ are defined on the triangular domain $\mathcal{T}=\{(x, \xi) \in [0,1]^2 \mid \xi \leq x\}$. They are bounded continuous functions defined by a set of hyperbolic PDEs given in~\cite{coron2013local}. The Volterra backstepping transformation \eqref{backstepping} is invertible~\cite{yoshida1960lectures} and the inverse transformation can be expressed as
\begin{equation}\label{backstepping_inverse}
\begin{aligned}
u(t, x) & = \alpha(t, x)+\int_0^x\left(L^{\alpha \alpha}(x, \xi) \alpha(t, \xi)+L^{\alpha \beta}(x, \xi) \beta(t, \xi)\right) \diff \xi, \\
v(t, x) & = \beta(t, x)+\int_0^x\left(L^{\beta \alpha}(x, \xi) \alpha(t, \xi)+L^{\beta \beta}(x, \xi) \beta(t, \xi)\right) \diff \xi,
\end{aligned}
\end{equation}
where the kernels $L^{\alpha \alpha}, L^{\alpha \beta}, L^{\beta a}$, and $L^{\beta \beta}$ are bounded continuous functions defined on $\mathcal{T}$. The dynamics of the system in the new coordinates are
\begin{equation}\label{new coordinates}
\left\{
\begin{aligned}
\alpha_t(t, x)+\lambda \alpha_x(t, x) & =0, \\
\beta_t(t, x)-\mu \beta_x(t, x) & =0,
\end{aligned}
\right.
\end{equation}
with boundary conditions
\begin{equation}\label{boundary}
\left\{
\begin{aligned}
\alpha(t, 0) & = q \beta(t, 0), \\
\beta(t, 1) & = \rho \alpha(t, 1) + \int_0^1 \left(N^\alpha(\xi) \alpha(t, \xi)+N^\beta(\xi) \beta(t, \xi) \right)\diff \xi,
\end{aligned}
\right.
\end{equation}
with
\begin{equation}
\begin{aligned}
N^\alpha(\xi) & = \rho L^{\alpha \alpha}(1, \xi) - L^{\beta \alpha}(1, \xi), \\
N^\beta(\xi) & = \rho L^{\alpha \beta}(1, \xi) - L^{\beta \beta}(1, \xi). 
\end{aligned}
\end{equation}

Using the method of characteristics on  \eqref{new coordinates} yields, for all $x \in[0,1]$ and $t>\tau$,
\begin{equation}
\label{eq:alpha-beta-characteristics}
\alpha(t, x)=q \beta\left(t-\frac{x}{\lambda}-\frac{1}{\mu}, 1\right), \quad \beta(t, x)=\beta\left(t-\frac{1-x}{\mu}, 1\right).
\end{equation}
Consequently, combining this with the boundary conditions \eqref{boundary}, we get
\begin{equation}\label{distributed delay}
\beta(t, 1)=q \rho \beta(t-\tau, 1) + \int_0^\tau N(\nu) \beta(t-\nu, 1) \diff \nu,
\end{equation}
where $\tau$ is defined by \eqref{tau} and $N$ is defined by
\begin{equation}\label{N}
N(\nu)=q \lambda N^\alpha\left(\lambda \nu-\frac{\lambda}{\mu}\right) \mathds{1}_{\left[\frac{1}{\mu}, \tau\right]}(\nu)+\mu N^\beta(1-\mu \nu) \mathds{1}_{\left[0, \frac{1}{\mu}\right]}(\nu).
\end{equation}

Consequently, $z(t)=\beta(t,1)$ is the solution of an IDE. Note also that, since the solution $(u, v)$ of \eqref{eq:hyperbolic_couple} belongs to $C^0([0, +\infty), H^1((0, 1), \mathbb R^2))$, the same is also true for the pair $(\alpha, \beta)$ defined in \eqref{backstepping}, thanks to the equations satisfied by the kernels $K^{u u}, K^{u v}, K^{v u}, K^{v v}$ from \cite[(3.30)--(3.37)]{coron2013local} and the regularity of these functions. Now, from \eqref{eq:alpha-beta-characteristics}, we have that $\beta(t-h, 1) = \beta(t, 1 - \mu h)$ for every $(t, h)$ with $0 \leq h \leq \frac{1}{\mu}$ and $t \geq h$, and it thus follows that $\beta(\cdot, 1) \in H^1((t - \frac{1}{\mu}, t), \mathbb R)$ for every $t \geq \frac{1}{\mu}$, which yields that $z \in H^1_{\mathrm{loc}}([0, +\infty), \mathbb R)$.

The following theorem, whose proof can be found in \cite[Theorem~6.1.3]{auriol2024contributions} or in~\cite{redaud2024domain}, shows how the $L^2$ stability properties of~$z$ relate to those of~$(\alpha, \beta)$ (and consequently to those of~$(u,v)$). 
\begin{theorem} \label{theorem_equiv_norm}
There exist two positive constants~$\kappa_0$ and~$\kappa_1$ such that, for every~$t>\tau$,
\begin{equation}
\label{eq_ineq_norm}
 \kappa_0 \lVert z_{[t]} \rVert^2_{L^2(-\frac{1}{\lambda},0)} \leq \lVert (\alpha(t,\cdot), \beta(t,\cdot)) \rVert^2_{(L^2(0,1))^2} \leq  \kappa_1 \lVert z_{[t]} \rVert^2_{L^2(-\tau,0)}.
\end{equation}
Moreover, the exponential stability of~$z_{[t]}$ in the sense of the~$L^2(-\tau, 0)$ norm is equivalent to the exponential stability of~$(\alpha,\beta)$ (or equivalently to $(u,v)$) in the sense of the~$L^2$ norm.
\end{theorem}

The fact that the norms are different on the two sides of \eqref{eq_ineq_norm} is related to the structure of the difference equation (see, for instance, the design of converse Lyapunov--Krasovskii functions~\cite{pepe2013converse}). The system~\eqref{distributed delay} can be seen as a \emph{comparison system} for the PDE system~\eqref{eq:hyperbolic_couple} (see, e.g.,~\cite{niculescu2001delay} and the references therein). In the rest of the paper, we will focus our stability analysis on the IDE~\eqref{distributed delay}.

\section{Main results}
\label{sec_main}

\subsection{Stability conditions for difference equation with distributed delay}

In light of the results presented in the previous section, we now focus on the stability analysis of the IDE
\begin{equation}\label{distributed delay0}
    z(t) = \xi z(t - \tau) + \int_0^\tau N(\nu) z(t - \nu) \diff\nu,
\end{equation}
where $\tau$ is a positive known delay, $\xi \in \R$, $N\colon [0, \tau] \to \mathbb R$ is an integrable function, and the unknown function is $z\colon [-\tau, +\infty) \to \mathbb R$. Even though the analysis of \eqref{distributed delay0} is motivated in this paper through its link with \eqref{eq:hyperbolic_couple}, we highlight that the stability analysis of IDEs has an interest on itself and in connection with more general time-delay systems (see, e.g., \cite[Chapter~9]{hale1993introduction}).

We assume that the initial data $z_{[0]} = z^0$ of $z$ is known, belongs to the space $H^1([-\tau,0],\mathbb{R})$, and verifies the compatibility condition $z^0(0)=\xi z^0(-\tau)+\int_0^\tau N(\nu)z^0(-\nu)\diff\nu$. A function~$z\colon [-\tau, \infty) \rightarrow \mathbb{R}$ is called a \emph{solution} of the IDE~\eqref{distributed delay0} with initial condition $z^0$ if~$z_{[0]} = z^0$ and if equation~\eqref{distributed delay0} is satisfied for every~$t \geq 0$. We will also assume here that
\begin{equation} \label{q_rho}
\lvert  \xi\rvert <1.
\end{equation}
This assumption is motivated by the fact that \eqref{distributed delay0} cannot be exponentially stable if $\lvert \xi\rvert>1$ \cite{henry1974linear,auriol2023robustification}, and amounts to assuming that the \emph{principal part} of the system~\eqref{distributed delay0} (that is, \eqref{distributed delay0} without the integral term corresponding to the distributed delay) is exponentially stable. However, due to the distributed delay term, system~\eqref{distributed delay0} may be unstable even under \eqref{q_rho}.

We analyze the stability properties of \eqref{distributed delay0} through spectral methods. Its characteristic function is the function $\Delta\colon \mathbb C \to \mathbb C$ defined by
\begin{equation}\label{characteristic equation}
\Delta(s) = 1-\xi \mathrm{e}^{-s \tau}-\int_0^\tau N(\nu) \mathrm{e}^{-s \nu} \diff \nu=0.
\end{equation}
The next result shows how the properties of the function~$\Delta$ relate to the stability properties of the IDE~\eqref{distributed delay0}.

\begin{lemma}[{\cite[Chapter~9, Theorem~3.5]{hale1993introduction}, \cite{henry1974linear}}]
The IDE~\eqref{distributed delay0} is ex\-po\-nen\-tially stable in $L^2$ norm if and only if there exists $\eta>0$ such that all solutions $s$ of the characteristic equation~\eqref{characteristic equation} satisfy $\real(s) \leq -\eta$.
\end{lemma}

We start our results by the following lemma, which provides a necessary condition for the stability of \eqref{distributed delay0} by studying the behavior of nonnegative real roots of $\Delta$.

\begin{lemma}\label{necessary_open}
The system \eqref{distributed delay0}  is not exponentially stable if \begin{equation*}\Delta(0)= 1-\xi -\int_0^\tau N(\nu)  \diff \nu\leq 0.
\end{equation*}
\end{lemma}
\begin{proof}
If $\Delta(0)=0$, then zero is a root of $\Delta$, and \eqref{distributed delay} is not exponentially stable. If $\Delta(0)<0$ then there exists at least one positive real root of $\Delta$  since
\[
\lim_{\substack{s \to +\infty \\ s \in \R}}\Delta(s)=1
\]
and $\Delta$ is continuous on $(0, \infty)$. Thus, system \eqref{distributed delay} is not exponentially stable.
\end{proof}

The following lemma presents an interesting property, which is a consequence of~\cite[Theorem~2.1]{hale2002strong}. 

\begin{lemma}\label{finite roots}
    Assume that \eqref{q_rho} is satisfied. Then, for all $s_0 > \frac{1}{\tau}\ln \lvert \xi \rvert$, $\Delta$ has a finite number of roots in  $\{s \in \mathbb C \mid \real(s) \geq s_0\}$.
\end{lemma}

This lemma implies that the function $\Delta$ can only have a finite number of roots on the imaginary axis. Let $\rho_1, \ldots, \rho_m$ be the positive zeros of $M(\omega)=\real\left(\Delta(i\omega)\right)$, repeated according to their multiplicities and ordered so that $0<\rho_m \leq \cdots \leq \rho_1$. The following theorem gives the number of roots of the characteristic function $\Delta$ which lie in $\{s \in \mathbb C \mid \real(s) > 0\}$, counted with their multiplicities.

\begin{theorem}\label{thm_open}
Assume that equation~\eqref{q_rho} is satisfied, that $\Delta$ has no roots on the imaginary axis, and that
\begin{equation}\label{nece_suf}
\int_0^\tau N(\nu)  \diff \nu< 1-\xi.
\end{equation}
Then the number of roots of the characteristic function $\Delta$ which lie in $\{s \in \mathbb C \mid \real(s)>0\}$, counted by multiplicity, is given by
\begin{equation} \label{2cond}
\Gamma :=  \sum_{j=1}^m(-1)^{j-1} \sign\left(S(\rho_j)\right),
\end{equation}
where $S\colon \mathbb R \to \mathbb R$ is the function given by $S(\omega)=\imag\left(\Delta(i\omega)\right)$.
\end{theorem}

\begin{proof}
For any $R>0$, let $C_R$ be the positively oriented contour defined by the curves $g_1$ and $g_2$, with
\[
g_1 \colon\left\{\begin{aligned}
\left[-\frac{\pi}{2}, \frac{\pi}{2}\right] & \to \mathbb C, \\
\theta & \mapsto R \mathrm e^{i \theta},
\end{aligned}\right. \quad\qquad g_2 \colon\left\{\begin{aligned}
\left[-R, R\right] & \to \mathbb C, \\
\omega & \mapsto -i\omega.
\end{aligned}\right.
\]
From Lemma~\ref{finite roots}, all zeros of $\Delta$ in $\{s \in \mathbb C \mid \real(s)>0\}$ are inside $C_R$, for sufficiently large $R$. By the argument principle, the number of zeros $n_0$ of $\Delta$ in $\{s \in \mathbb C \mid \real(s)>0\}$, counted with their multiplicities, is given by
\begin{equation}\label{prin_gene}
n_0=\frac{1}{2 \pi i} \oint_{C_R} \frac{\Delta^{\prime}(s)}{\Delta(s)} \diff s=\frac{1}{2 \pi i} \int_{g_1} \frac{\Delta^{\prime}(s)}{\Delta(s)} \diff s+\frac{1}{2 \pi i} \int_{g_2} \frac{\Delta^{\prime}(s)}{\Delta(s)} \diff s.
\end{equation}
We now focus on computing the different integral terms.
\medskip

\noindent\uline{Value of the integral over $g_1$ in \eqref{prin_gene}:} 
From the Riemann--Lebesgue lemma, we have that
\[
\lim_{\substack{\lvert s\rvert \to +\infty \\ \real(s) \geq 0}} \int_0^\tau N(\nu) \mathrm{e}^{-s \nu} \diff \nu = 0.
\]
Then, for all $\epsilon>0$, there exists $R_0>0$  such that, for all $R\geq R_0$ and $\theta \in \left[-\frac{\pi}{2}, \frac{\pi}{2}\right]$,
\begin{equation}
\real(\Delta(R\mathrm{e}^{i\theta}))\geq 1-\lvert  \xi\rvert-\epsilon.
\end{equation}
In particular, for $\epsilon=\frac{1-\lvert  \xi \rvert}{2}$, we have $ \real(\Delta(R\mathrm{e}^{i\theta}))>0$ for sufficiently large $R$ and all $\theta \in \left[-\frac{\pi}{2}, \frac{\pi}{2}\right]$. Then, for sufficiently large $R$, the function $s \mapsto \Ln(\Delta(s)) = \ln \lvert \Delta(s) \rvert + i \arg(\Delta(s))$ is an analytic function in a neighborhood of $g_1$, with $\frac{\mathrm d}{\mathrm d s} \Ln(\Delta(s)) = \frac{\Delta^\prime(s)}{\Delta(s)}$, and thus
\begin{align}
\frac{1}{2 \pi i} \int_{g_1} \frac{\Delta^{\prime}(s)}{\Delta(s)} \diff s  = \frac{\Ln\Delta(i R) - \Ln\Delta(-i R)}{2 \pi i} 
 = \frac{1}{\pi} \arg(\Delta(i R)), \label{g1_gene}
\end{align}
where we used the fact that $\Delta(-i R) = \overline{\Delta(i R)}$.

\medskip

\noindent\uline{Value of the integral over $g_2$ in \eqref{prin_gene}:}
Using the fact that 
\[\Delta(-i \omega) = \overline{\Delta(i \omega)} \quad \text{for every }\omega \in \mathbb R,
\]
we obtain that $\Delta^\prime(-i\omega) = \overline{\Delta^\prime(i \omega)}$ for every $\omega \in \mathbb R$, which implies
\[
\frac{1}{2 \pi i} \int_{g_2} \frac{\Delta^{\prime}(s)}{\Delta(s)} \diff s = -\frac{1}{\pi} \int_0^R \real\left(\frac{\Delta^\prime(i \omega)}{\Delta(i \omega)}\right) \diff \omega.
\]
Since $\Delta$ has no roots on the imaginary axis, we have $\Delta(i\omega)=A(\omega)\mathrm{e}^{i\phi(\omega)}$ for some differentiable functions $A\colon \mathbb R \to \mathbb R_+^\ast$ and $\phi\colon \mathbb R \to \mathbb R$ with $\phi(R) = \arg(\Delta(i R))$. Hence
\begin{equation*}
\frac{\Delta'(i\omega)}{\Delta(i\omega)}=-i\frac{A'(\omega)}{A(\omega)}+\phi'(\omega),
\end{equation*}
and we deduce that
\[
\frac{1}{2 \pi i} \int_{g_2} \frac{\Delta^{\prime}(s)}{\Delta(s)} \diff s = \frac{1}{\pi}(\phi(0) - \phi(R)).
\]
According to Condition~\eqref{nece_suf}, we have $M(0)>0$ and $S(0)=0$, where $M\colon \omega\in\R\mapsto \real\left(\Delta(i\omega)\right)$ and   $S\colon \omega\in\R\mapsto \imag\left(\Delta(i\omega)\right)$. Then, as shown in \cite[Section~3.7]{hassard1997counting}, we can prove that
\[\phi(0)=\pi\sum_{j=1}^m(-1)^{j-1} \sign\left(S(\rho_j)\right).\]
Thus
\begin{equation}\label{int_g2}
 \frac{1}{2 \pi i} \int_{g_2} \frac{\Delta'(s)}{\Delta(s)} \diff s=\sum_{j=1}^m(-1)^{j-1} \sign\left(S(\rho_j)\right)-\frac{\arg(\Delta(i R))}{\pi} .
\end{equation}
Combining~\eqref{g1_gene} and \eqref{int_g2}, we finally obtain that
\begin{align}\label{prin_final}
\frac{1}{2 \pi i} \oint_{C_R} \frac{\Delta^{\prime}(s)}{\Delta(s)} \diff s=\sum_{j=1}^m(-1)^{j-1} \sign\left(S(\rho_j)\right).
\end{align}
Consequently, the number of roots of the characteristic function $\Delta$ which lie in $\real(s)>0$, counted by multiplicity, is given by
\begin{equation} 
\sum_{j=1}^m(-1)^{j-1} \sign\left(S(\rho_j)\right),
\end{equation}
as required.
\end{proof}

Combining Lemma~\ref{necessary_open}, Theorem~\ref{thm_open}, and Lemma~\ref{finite roots}, we obtain at once the following exponential stability result for \eqref{distributed delay0}. 

\begin{corollary}
\label{coro:stability}
Assume that equation~\eqref{q_rho} is satisfied and that $\Delta$ has no roots on the imaginary axis. Then, the system \eqref{distributed delay0} is exponentially stable if and only if \eqref{nece_suf} holds and $\Gamma=0$.
\end{corollary}

\begin{remark}
Note that the use of this result requires finding the roots of $\Delta$ on the imaginary axis. This is equivalent to finding the common roots of the functions $M$ and $S$, where $M\colon \omega \in \mathbb{R} \mapsto \real(\Delta(i\omega))$ and $S\colon \omega \in \mathbb{R} \mapsto \imag(\Delta(i\omega))$. For this purpose, there are many numerical solvers available, such as those in numerical libraries like Scipy (with functions such as \textbf{scipy.optimize.root} and \textbf{scipy.optimize.newton}) and Matlab (with functions like \textbf{fsolve} and \textbf{roots}). These tools and methods enable efficient handling of the task of finding roots of real-valued functions defined on the set of real numbers.
\end{remark}

\begin{remark}
One might wonder whether condition \eqref{nece_suf} is sufficient to guarantee that the function $\Delta$ has no root on the imaginary axis. To see that this is not the case, consider the function $N$ given by
\[N(\nu)=-\frac{3\pi^2}{16-4\pi}\nu + \frac{\pi^2+2\pi}{16 - 4\pi},\]
with $\tau=1$ and $\xi=\frac{1}{2}$. Then, we obtain
\[\int_0^\tau N(\nu)  \diff \nu=\frac{\pi}{8}<\frac{1}{2},\]
which shows that condition \eqref{nece_suf} is satisfied. On the other hand, we compute
\[
\Delta(i\tfrac{\pi}{2}) = 1 + \frac{i}{2} - \int_0^1 N(\nu) \mathrm e^{-i \nu \frac{\pi}{2}} \diff\nu = 0,
\]
showing that $i\frac{\pi}{2}$ is a root of $\Delta$ (and so is its complex conjugate $-i\frac{\pi}{2}$). 
Figure~\ref{fig:_1} shows the roots\footnote{Computations of the roots were performed with Python's \texttt{cxroots} package \cite{parini2018cxroots}, and the search of roots was limited to the rectangle $\{s \in \mathbb{C} \mid \lvert \real(s) \rvert \leq 5,\lvert \imag(s)\rvert \leq 20\}$.} of $\Delta$ and we can see the presence of roots of $\Delta$ on the imaginary axis.
\begin{figure}[ht]
\centering
\includegraphics[width=0.7\textwidth]{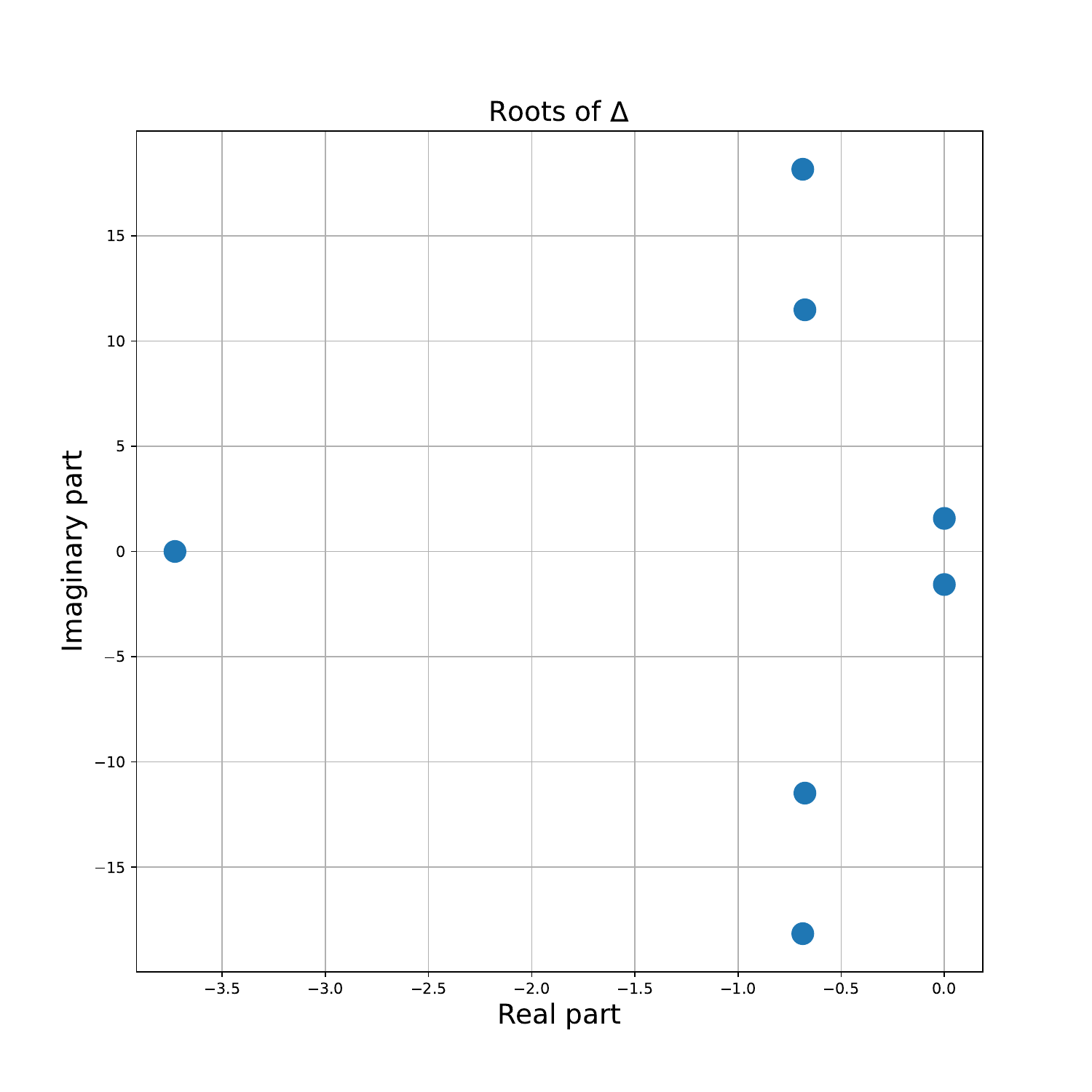}
\caption{Roots of $\Delta$ with $N(\nu)=-\frac{3\pi^2}{16-4\pi}\nu + \frac{\pi^2+2\pi}{16 - 4\pi}$.}
\label{fig:_1}
\end{figure}
\end{remark}

\subsection{Application to the stability analysis of the system~\eqref{eq:hyperbolic_couple}}

Since the PDE system \eqref{eq:hyperbolic_couple} has equivalent stability properties to those of~\eqref{distributed delay0} with $\xi=\rho q$, we directly obtain the following corollary

\begin{corollary} \label{result:eqhyp}
Assume that $\lvert  q\rho\rvert<1$ and that the characteristic equation $\Delta$ (defined by equation~\eqref{characteristic equation} with $\xi=\rho q$) has no roots on the imaginary axis. Then, the system \eqref{eq:hyperbolic_couple} is exponentially stable if and only if equation \eqref{nece_suf} holds and $\Gamma=0$.
\end{corollary} 

\section{Comparison with other criteria and numerical validation}\label{sec_comparaison}

Our next goal is to compare the stability criterion from Corollary~\ref{result:eqhyp} to similar results in the literature. More precisely, we will compare our results to the stability criteria from \cite{bastin2016stability, saba2019stability} for the case where the in-domain coupling terms $\sigma^+$ and $\sigma^-$ are constant. In this particular situation, it is shown in \cite{saba2019stability} that the function $N$ from \eqref{N} is given by
\begin{equation}
\label{eq:explicit-N-Bessel}
N(\nu) = \left(\frac{a}{\tau} + \frac{d(\nu)}{\tau^2}\right) J_0\left(2 \sqrt{h(\nu)}\right) + \frac{d(\nu)}{\tau^2} J_2\left(2 \sqrt{h(\nu)}\right), \quad \nu \in [0, \tau],
\end{equation}
where
\[
\begin{aligned}
h(\nu) & = \frac{R}{\tau^2}\nu(\tau-\nu), & \quad d(\nu) & = R(\tau-\nu b), \\
a & = \frac{q}{\mu}\sigma^- + \frac{\rho}{\lambda}\sigma^+, & R & = \frac{\sigma^+\sigma^-}{\lambda\mu}, &\quad b & = 1 + q\rho,
\end{aligned} 
\]
and $J_0$ and $J_2$ are Bessel functions of the first kind (see, e.g., \cite{olver2010nist}). Using the series expansion of Bessel functions from \cite[(10.2.2)]{olver2010nist}, the above expression of $N$ can be rewritten as
\begin{equation}
\label{eq:explicit-N-series}
N(\nu) = \left(\frac{a}{\tau} + \frac{d(\nu)}{\tau^2}\right) \sum_{p=0}^{\infty} \frac{(-1)^p}{(p!)^2} \left(h(\nu)\right)^{p} + \frac{d(\nu)}{\tau^2} \sum_{p=0}^{\infty} \frac{(-1)^p}{p! (p + 2) !} \left(h(\nu)\right)^{p + 1},  \quad \nu \in[0, \tau].
\end{equation}

Let us now recall the stability criteria from \cite{bastin2016stability, saba2019stability} with which we will compare our Corollary~\ref{result:eqhyp}. We first present the Lyapunov-based exponential stability condition for linear systems of balance laws from~\cite[Theorem~5.4]{bastin2016stability}.

\begin{proposition}[{\cite[Theorem~5.4]{bastin2016stability}}]
\label{prop:criterion-bastin-coron}
Assume that $\sigma^+$ and $\sigma^{-}$ are constants and define
\[M=\begin{pmatrix}
0&-\sigma^+\\-\sigma^{-}&0
\end{pmatrix}, \quad \mathbf{K}=\begin{pmatrix}
0&q\\\rho&0
\end{pmatrix}, \quad \text{ and } \quad \Lambda=\begin{pmatrix}
\lambda&0\\0&-\mu
\end{pmatrix}.\]
If there exists a $2\times 2$ positive diagonal real matrix $P$  such that
\[
M^{\top} P+P M \text { is positive semi-definite }
\]
and
\[
\left\lVert \delta \mathbf{K} \delta^{-1}\right\rVert<1
\]
with $\delta \triangleq \sqrt{P\lvert \Lambda\rvert}$ (the operations here being understood component-wise), then system \eqref{eq:hyperbolic_couple} is exponentially stable for the $L^2$ norm.
\end{proposition}

Using the fact that the stability of the system is equivalent to having $\lvert \Delta(s)\rvert>0$ for all $s \in \mathbb{C}$ such that $\real(s)\geq 0$, the following sufficient stability condition was obtained in~\cite{saba2019stability}.

\begin{proposition}[{\cite[Proposition~3]{saba2019stability}}]
\label{prop:criterion-jean-et-al}
Assume that the coefficients $\sigma^+$ and $\sigma^{-}$ are constants. If the constant parameters of system \eqref{eq:hyperbolic_couple} satisfy either of the following set of inequalities
\begin{enumerate}
\item $\sigma^+ \sigma^{-} \geq 0$, $\rho q \geq 0$, and 
\[\lvert a\rvert+\lvert R\rvert\left(\frac{1}{1+\lvert \rho q\rvert}-\frac{1-\lvert \rho q\rvert}{2}\right)<1-\lvert \rho q\rvert;\]
\item $\sigma^+ \sigma^{-} \geq 0$, $\rho q < 0$, and 
\[\lvert a\rvert+\lvert R\rvert \frac{1+\lvert \rho q\rvert}{2}<1-\lvert \rho q\rvert;\]
\item $\sigma^+ \sigma^{-} < 0$, $\rho q \geq 0$, and
\[\lvert a\rvert I_0(\sqrt{\lvert R\rvert})+\lvert R\rvert\left(\frac{1}{1+\lvert \rho q\rvert}-\frac{1-\lvert \rho q\rvert}{2}\right) \left[I_0(\sqrt{\lvert R\rvert})-I_2(\sqrt{\lvert R\rvert})\right]<1-\lvert \rho q\rvert;\]
\item $\sigma^+ \sigma^{-} <0$, $\rho q<0$, and
\[ \lvert a\rvert I_0(\sqrt{\lvert R\rvert})+\lvert R\rvert \frac{1+\lvert \rho q\rvert}{2} \left[I_0(\sqrt{\lvert R\rvert})-I_2(\sqrt{\lvert R\rvert})\right]<1-\lvert \rho q\rvert;\]
\end{enumerate}
where $I_0$ and $I_2$ are modified Bessel functions of the first kind (see, e.g., \cite[Section~10.25]{olver2010nist}), then system \eqref{eq:hyperbolic_couple} is exponentially stable for the $L^2$ norm.
\end{proposition}

Let us also mention that, using recent ISS results and a small-gain property, another sufficient stability condition can be found in~\cite{karafyllisinput}. When applied to system~\eqref{eq:hyperbolic_couple}, this condition requires the existence of a constant $K>0$, with $(\lvert \rho\rvert+\lvert q\rvert)\mathrm{e}^{-K}<1$ such that
\begin{equation}
\label{eq:criterion-iss}
   \left(\sqrt{\frac{\mathrm{e}^{2K}-\mathrm{e}^K}{K\lambda}\lvert \sigma^-\rvert}+\sqrt{\lvert \rho\rvert}\right)\left(\sqrt{\frac{\mathrm{e}^{2K}-\mathrm{e}^K}{K\mu}\lvert \sigma^+\rvert}+\sqrt{\lvert q\rvert}\right)<1.
\end{equation}
However, it has been shown in~\cite{saba2019stability} that this condition is not satisfied for numerous examples (and in particular, the ones we consider below). We refer to \cite[Section~V]{saba2019stability} for a more detailed comparison between \cite[Proposition~3]{saba2019stability} and \eqref{eq:criterion-iss}.

\begin{table}[ht] 
\centering
\footnotesize
\caption{Comparison with other criteria}
\label{tab:comparison}
\begin{tabular}{@{\hspace*{0.01\textwidth}} >{\centering} m{0.34\textwidth} @{\hspace*{0.02\textwidth}} >{\centering} m{0.22\textwidth} @{\hspace*{0.02\textwidth}} >{\centering} m{0.15\textwidth} @{\hspace*{0.02\textwidth}} >{\centering} m{0.15\textwidth} @{\hspace*{0.01\textwidth}}}
\toprule
Example of system \eqref{eq:hyperbolic_couple} \\ 
$\bigl(\sigma^+,\allowbreak \sigma^-,\allowbreak \frac{1}{\lambda},\allowbreak \frac{1}{\mu},\allowbreak \rho,\allowbreak q\bigr)$
 & Conditions of Corollary~\ref{result:eqhyp} & Conditions of Proposition~\ref{prop:criterion-bastin-coron} & Conditions of Proposition~\ref{prop:criterion-jean-et-al} \tabularnewline
\midrule
$(1.1,\allowbreak 0.4,\allowbreak 1,\allowbreak 1.2,\allowbreak 0.4,\allowbreak -0.5)$ & Satisfied & Not satisfied & Satisfied \tabularnewline
$(-0.8,\allowbreak 0.7,\allowbreak 1,\allowbreak 1.2,\allowbreak 0.4,\allowbreak 0.25)$ & Satisfied & Satisfied & Satisfied \tabularnewline
$(1.3,\allowbreak -0.95,\allowbreak 1.8,\allowbreak 0.44,\allowbreak 0.45,\allowbreak 0.25)$ & Satisfied & Not satisfied & Satisfied \tabularnewline
$(1.3,\allowbreak -1.2,\allowbreak 1.8,\allowbreak 1.5,\allowbreak 0.45,\allowbreak 0.25)$ & Satisfied & Satisfied & Not satisfied \tabularnewline
$(2.3,\allowbreak -3.5,\allowbreak 0.8,\allowbreak 1.1,\allowbreak 0.5,\allowbreak -0.7)$ & Satisfied & Not satisfied & Not satisfied \tabularnewline
\bottomrule
\end{tabular}
\end{table}

To compare our criterion with Propositions~\ref{prop:criterion-bastin-coron} and \ref{prop:criterion-jean-et-al} above, we selected five sets of parameters $\sigma^+$, $\sigma^-$, $\lambda$, $\mu$, $\rho$, and $q$, given in the first column of Table~\ref{tab:comparison}. We selected the same sets of parameters as in \cite[Table~II]{saba2019stability} as they provide showcase scenarios for which the criteria specified in \cite[Theorem~5.4]{bastin2016stability} and \cite[Proposition~3]{saba2019stability} are not always satisfied.

To verify our criterion from Corollary~\ref{result:eqhyp}, we first verify that the inequalities $\lvert q \rho \rvert < 1$ and \eqref{nece_suf} are satisfied. Then, to compute the quantity $\Gamma$ from \eqref{2cond}, we first compute numerically the positive zeros $\rho_1, \dotsc, \rho_m$ of the function $M$ given by $M(\omega) = \real (\Delta(i \omega))$. In all the examples from Table~\ref{tab:comparison}, the inequalities $\lvert q \rho \rvert < 1$ and \eqref{nece_suf} are satisfied, and $M$ turns out to have no real roots, which immediately implies that $\Delta$ has no zeros on the imaginary axis, $\Gamma = 0$, and hence system \eqref{eq:hyperbolic_couple} is exponentially stable by Corollary~\ref{result:eqhyp}. Our verifications of these conditions were performed using MATLAB, using the built-in Bessel function and numerical integration. As an example, with the parameters $\left(\sigma^+, \sigma^-, \frac{1}{\lambda}, \frac{1}{\mu}, \rho, q\right) = (2.3, -3.5, 0.8, 1.1, 0.5, -0.7)$ (corresponding to the last row of Table~\ref{tab:comparison}), we get $\int_0^\tau N(\nu) \diff\nu \approx 1.3143 < 1 + 0.5 \times 0.7 = 1.35$ and, from the plot of the function $M$ shown in Figure~\ref{M_plot}, we observe that $M$ has no positive real root (note that $M(0)\approx 0.0357 \neq 0$). Therefore, $\Gamma = 0$ and, by Corollary~\ref{coro:stability}, the system is exponentially stable for these parameters.

\begin{figure}[ht]
\centering
\includegraphics[width=0.7\textwidth]{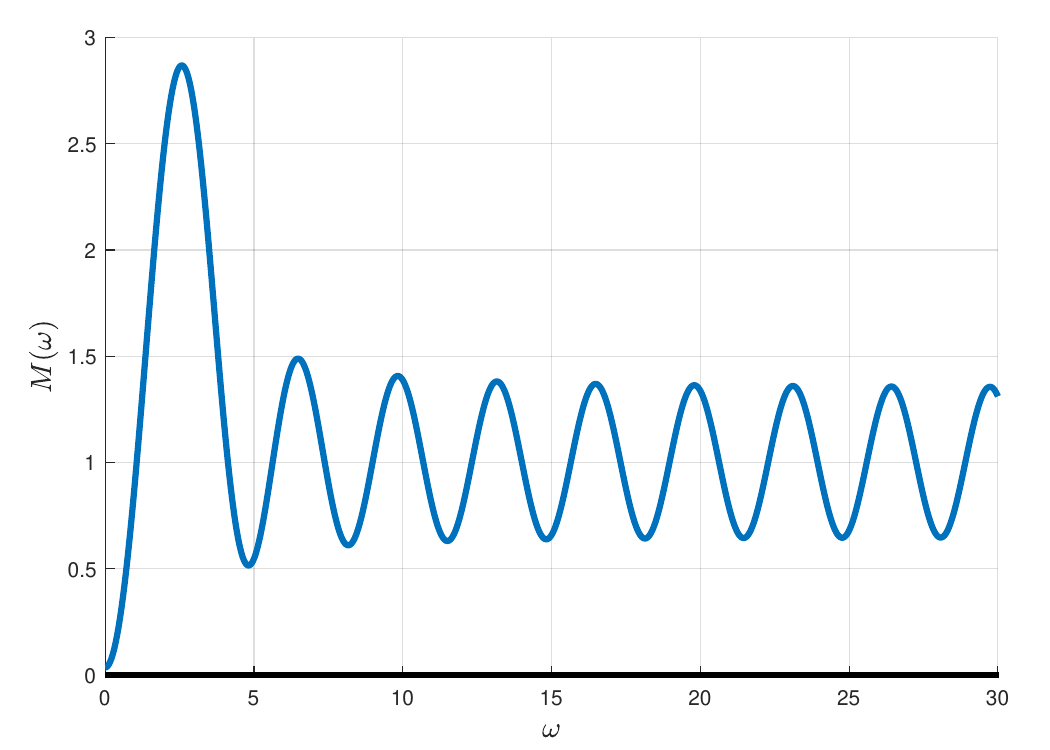}
\caption{Graph of the function $M$ when $\left(\sigma^+, \sigma^-, \frac{1}{\lambda}, \frac{1}{\mu}, \rho, q\right)=(2.3,-3.5,0.8,1.1,0.5,-0.7)$.}
\label{M_plot}
\end{figure}

To illustrate this exponential stability property, we represent in Figure~\ref{fig_L2_norm} the evolution in time of the $L^2$ norm in space of the solution $(u, v)$ of~\eqref{eq:hyperbolic_couple} with the parameters $\left(\sigma^+, \sigma^-, \frac{1}{\lambda}, \frac{1}{\mu}, \rho, q\right) = (2.3, -3.5, 0.8, 1.1, 0.5, -0.7)$.
We simulated our system with a time horizon of $500$, a standard first-order upwind scheme for the discretization of the space derivative, a time step of $\Delta t = 10^{-4}$, and different space discretization steps $\Delta x_u$ and $\Delta x_v$ for $u$ and $v$ in such a way that the CFL conditions $\frac{\lambda \Delta t}{\Delta x_u} \leq 1$ and $\frac{\mu \Delta t}{\Delta x_v} \leq 1$ are satisfied as close as possible to the equality. The initial condition $u_0$ was chosen as a constant equal to $1$, while the initial condition $v_0$ was chosen as the unique affine function such that the zero-order compatibility condition \eqref{compatibility_condition_u_v} is fulfilled. The result of the simulation, presented in Figure~\ref{fig_L2_norm}, show the exponential convergence to the origin of the considered solution in $L^2$ norm, as expected.

\begin{figure}[htp]
\centering
\resizebox{0.6\textwidth}{!}{\input{Figures/norm_L2_solution.pgf}}
\caption{$L^2$ norm of the solution $(u,v)$ of~\eqref{eq:hyperbolic_couple} with parameters $\left(\sigma^+, \sigma^-, \frac{1}{\lambda}, \frac{1}{\mu}, \rho, q\right)=(2.3,\allowbreak -3.5,\allowbreak 0.8,\allowbreak 1.1,\allowbreak 0.5,\allowbreak -0.7)$.}
\label{fig_L2_norm}
\end{figure}

As the conditions from Corollary~\ref{result:eqhyp} are necessary and sufficient, we are able to conclude on the stability of systems for which the criteria from Propositions~\ref{prop:criterion-bastin-coron} or \ref{prop:criterion-jean-et-al} are not satisfied, as illustrated in Table~\ref{tab:comparison}. In addition, despite the potential difficulties in the computation of the zeros of $M$ and of the value of $\Gamma$, a numerical test of our necessary and sufficient condition from Corollary~\ref{result:eqhyp} seems less complex and easier to verify than the numerical test for the necessary and sufficient condition in \cite[Condition~(31)]{saba2019stability}. 

\section{Numerical explorations of truncations of the function $N$} \label{sec_approximation}

One of the difficulties in the analysis of the case where $\sigma^+$ and $\sigma^-$ are constant is that the function $N$ from \eqref{eq:explicit-N-Bessel} depends on the Bessel functions $J_0$ and $J_2$, which are transcendental functions. On the other hand, when the function $N$ from \eqref{characteristic equation} is a polynomial, the function $\Delta$ becomes a quasipolynomial (more precisely, a quasipolynomial divided by a power of $s$), whose properties are easier to analyze. A natural idea to simplify the analysis in the case where $\sigma^+$ and $\sigma^-$ are constant is thus to approximate $N$ by truncating the series in \eqref{eq:explicit-N-series} and study the quasipolynomial $\Delta$ obtained by this procedure. The aim of this section is to provide numerical explorations around this idea.

For $p \in \N$, let $N_p$ denote the function obtained by truncating the sums in \eqref{eq:explicit-N-series} so that each sum has only its first $p+1$ terms, i.e.,
\begin{equation*}
N_p(\nu) = \left(\frac{a}{\tau} + \frac{d(\nu)}{\tau^2}\right) \sum_{p^\prime=0}^{p} \frac{(-1)^{p^\prime}}{(p^\prime !)^2} \left(h(\nu)\right)^{p^\prime} + \frac{d(\nu)}{\tau^2} \sum_{p^\prime=0}^{p} \frac{(-1)^{p^\prime}}{p^\prime ! (p^\prime + 2) !} \left(h(\nu)\right)^{p^\prime + 1},  \quad \nu \in[0, \tau].
\end{equation*}
Note that
\begin{equation}\label{expression N0}
N_0(\nu)=\frac{R^2b}{2\tau^4}\nu^3-\frac{R^2(1+b)}{2\tau^3}\nu^2+\left(\frac{R^2}{2\tau^2}-\frac{Rb}{\tau^2}\right)\nu+\frac{a+R}{\tau}
\end{equation}
and that, for $p \in \N$, we have the recurrence relation
\begin{equation}\label{expression Np+1}
N_{p+1}(\nu)= N_p(\nu) + f_p(\nu),
\end{equation}
where
\begin{equation}\label{expression fp}
f_p(\nu) = (-1)^{p+1}\Bigg(\left(\frac{a}{\tau}+\frac{d(\nu)}{\tau^2}\right) \frac{(h(\nu))^{p+1}}{((p+1)!)^2}+\frac{(h(\nu))^{p+2}d(\nu)}{\tau^2(p+1)!(p+3)!}\Bigg)
\end{equation}
and  $h$, $d$, $a$, $R$, and $b$ are defined as in Section~\ref{sec_comparaison}.

For $p \in \N$, let $\Delta_p$ be the characteristic function associated with $N_p$, i.e., we set
\begin{equation}
\label{eq:Delta-p}
\Delta_{p}(s)= 1-q\rho \mathrm{e}^{-s \tau}-\int_0^\tau N_{p}(\nu) \mathrm{e}^{-s \nu} \diff \nu.
\end{equation}
Rewriting $N_p$ as $N_p(\nu)=\sum_{k=0}^{2p+3}a_k\nu^k$ for suitable coefficients $a_0, \dotsc, a_{2 p + 3}$, then, for $s\neq 0$, we can rewrite $\Delta_p$ as $\Delta_p(s)=P_0(s)+P_1(s)\mathrm{e}^{-s \tau}$ with
\begin{equation}
\label{eq:P0P1truncation}
\left\{
\begin{aligned}
P_0(s) & = \frac{s^{2p+4}-\sum^{2p+3}_{k=1}(2p+3-k)!a_{2p+3-k}s^k-(2p+3)!a_{2p+3}}{s^{2p+4}}, \\
P_1(s) & = \frac{-q\rho s^{2p+4}+\sum^{2p+3}_{k=0}s^k N_p^{(2p + 3 - k)}(\tau)}{s^{2p+4}}.
\end{aligned}
\right.
\end{equation}

\subsection{Investigation of imaginary roots}

As Corollary~\ref{result:eqhyp} requires one to check whether the function $\Delta$ admits roots on the imaginary axis, we now describe in more detail some conditions for the existence of imaginary roots of the function $\Delta_p$ obtained by the truncation $N_p$ of $N$.

\subsubsection{The case $p=0$}

Let us obtain a necessary condition for the existence of nontrivial roots in the imaginary axis for $\Delta_0$. For $\omega \in \R$, the imaginary number $i \omega$ is a root of $\Delta_0$ if and only if
\[
P_0(i\omega)\mathrm{e}^{i\omega\tau} = - P_1(i\omega),
\]
where $P_0$ and $P_1$ are as in \eqref{eq:P0P1truncation} with $p = 0$. The latter condition implies that
\[
\lvert  P_0(i\omega)\rvert^2=\lvert  P_1(i\omega)\rvert^2,
\]
which, using the expressions from \eqref{eq:P0P1truncation}, is equivalent to
\begin{equation}
\label{egalté P_P1}
\omega^4\Bigg((1-q^2\rho^2)\omega^{4}+\omega^{2}\left( a_0^2+2a_1-N_0(\tau)^2-2q\rho N_0'(\tau)\right)-12a_3\Delta_0(0)\Bigg)=0.
\end{equation}
If we assume that $N_0$ satisfies \eqref{nece_suf} then, $\Delta_0(0)>0$. Consequently, we obtain that the solutions of~\eqref{egalté P_P1} verify $\omega\neq 0$. Moreover, according to \eqref{q_rho}, we have $a_3\geq 0$. Then, we obtain from \eqref{egalté P_P1} that
\begin{equation}
\label{eq:explicit-omega-0}
\omega=\pm\sqrt{\frac{-\left( a_0^2+2a_1-N_0(\tau)^2-2q\rho N_0'(\tau)\right)+\sqrt{D}}{2(1-q^2\rho^2)}}
\end{equation}
with
\[
D=\left( a_0^2+2a_1-N_0(\tau)^2-2q\rho N_0'(\tau)\right)^2+48a_3(1-q^2\rho^2)\Delta_0(0).
\]
This means that, if \eqref{nece_suf} is satisfied for $N_0$, then $\Delta_0$ has at most two roots on the imaginary axis, and these roots can only be the ones from \eqref{eq:explicit-omega-0}. Depending on the parameter values, it may have either exactly two imaginary roots or no imaginary roots.

\subsubsection{The case $p=1$}

Proceeding as previously, a straightforward computation shows that, if $\omega\in \R$ is such that $\Delta(i \omega) = 0$, then necessarily
\begin{multline}
\label{or1_egalté P_P1}
\omega^6 \Biggl( \Bigl(1-q^2\rho^2\Bigr) \omega^{6} + \Bigl(a_0^2+2a_1-N_1(\tau)^2-2q\rho N_1'(\tau)\Bigr) \omega^{4} \\
 + \Bigl(a_1^2-12a_3-4a_0a_2-N_1'(\tau)^2+2q\rho N_1^{'''}(\tau)+2N_1(\tau)N_1{''}(\tau)\Bigr)\omega^2 + 240\Delta_1(0)\Biggr)=0,
\end{multline}
with
\begin{equation}
\begin{aligned}
a_0 & = \frac{a+R}{\tau}, & \quad a_1 & = -\frac{R^2+2R(a+b)}{2\tau^2}, \\
a_2 & = \frac{6aR+3R^2(b+1)-R^3}{6\tau^3}, & a_3 & = \frac{-3R^2b+R^3(b+2)}{6\tau^4}.
\end{aligned}
\end{equation}

Assume that \eqref{q_rho} and \eqref{nece_suf} are satisfied with $N$ replaced by the truncation $N_1$. If $\delta<0$, then $\Delta$ has at most 2 roots on the imaginary axis, where
\begin{equation}
   \left\{ \begin{aligned}
   &\delta=18ABCE-4B^3E+B^2C^2-4AC^3-27A^2E^2,\\
    &A=1-q^2\rho^2,\, B= a_0^2+2a_1-N_1(\tau)^2-2q\rho N_1'(\tau),\\
    &C=a_1^2-12a_3-4a_0a_2-N_1'(\tau)^2+2q\rho N_1^{'''}(\tau)+2N_1(\tau)N_1{''}(\tau),\\
    &E=240\Delta_1(0).
    \end{aligned}\right.
\end{equation} In addition, if $Q\geq 0$, then $\Delta$ has no root on the imaginary axis, where
\begin{equation}
    Q=\frac{B}{27 A}\left(\frac{2 B^2}{A^2}-\frac{9 C}{A}\right)+\frac{E}{A} .
\end{equation}

With these first two truncations, we can conclude that finding the roots of $\Delta_p$ on the imaginary axis amounts to searching for the roots of a polynomial with real coefficients and finite degrees, which constitutes an algebraic problem.

\subsection{Stability of the truncated system}

We now numerically compare the stability properties of \eqref{distributed delay} when $N$ is given by \eqref{eq:explicit-N-Bessel} or by one of its truncations $N_p$, $p \in \mathbb N$.

\begin{figure}[htp]
\centering
\includegraphics[width=0.7\textwidth]{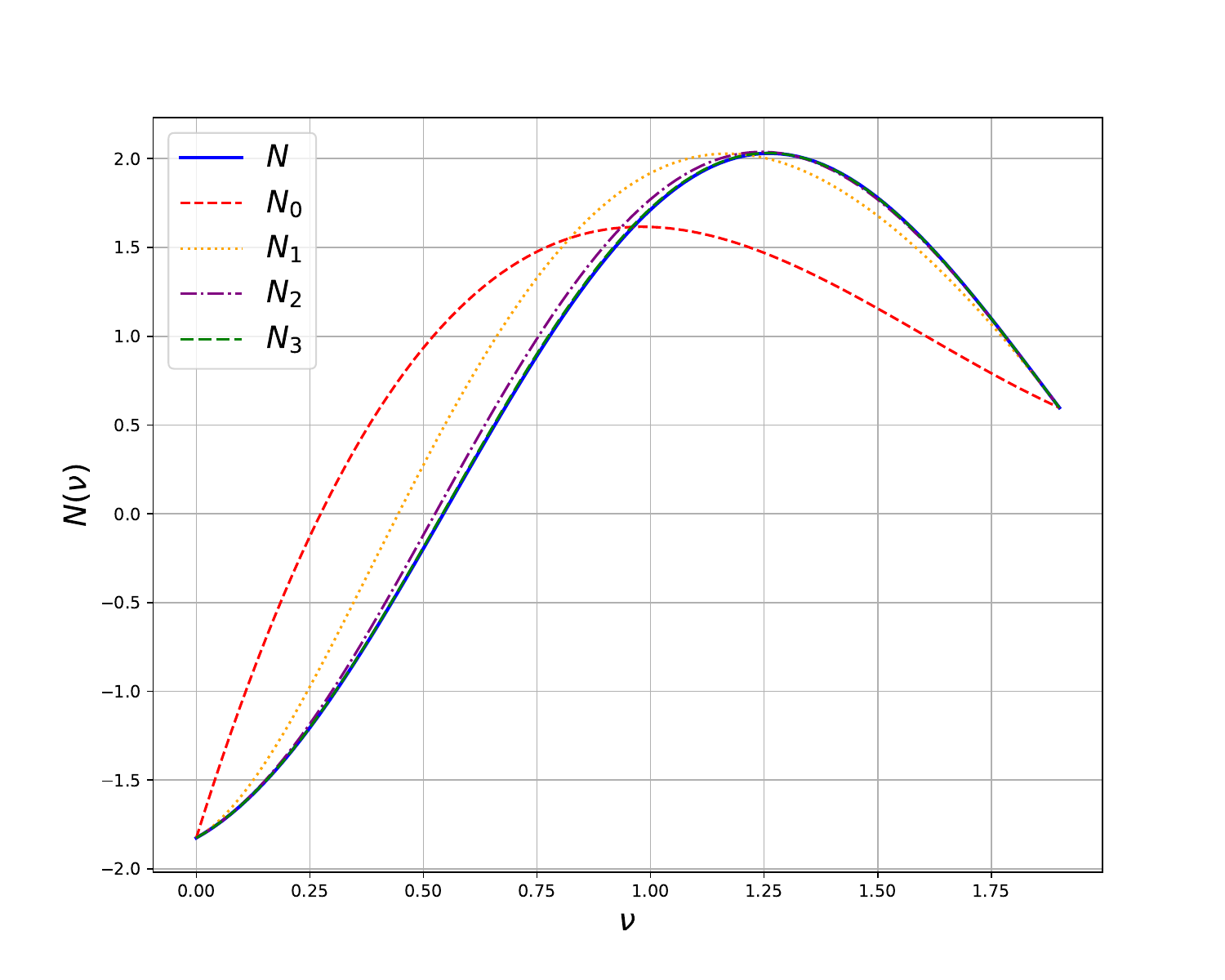}
\caption{$N$, $N_0$, $N_1$, $N_2$, and $N_3$ with $\Bigl(\sigma^+, \sigma^-, \frac{1}{\lambda}, \frac{1}{\mu}, \rho, q\Bigr)=(2.3,-3.5,0.8,1.1,0.5,-0.7)$.}
\label{fig:_2}
\end{figure}

\begin{figure}[htp]
\centering
\includegraphics[width=0.7\textwidth]{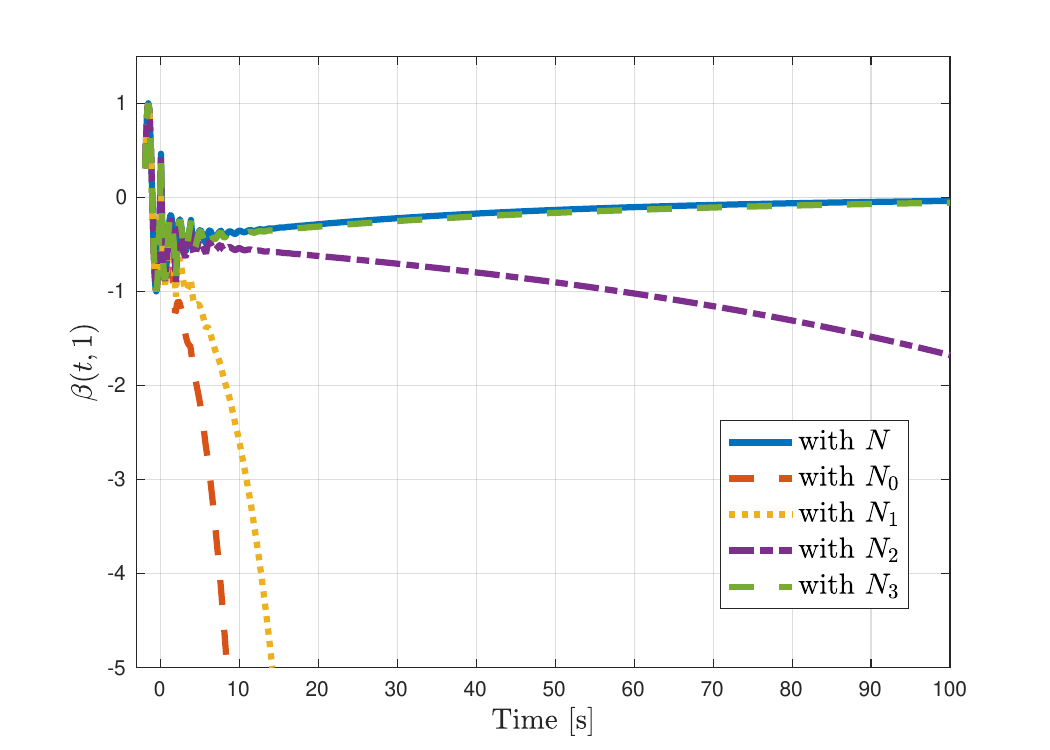}
\caption{Solution of \eqref{distributed delay} with $N$, $N_0$, $N_1$, $N_2$, and $N_3$ with $\Bigl(\sigma^+,\allowbreak \sigma^-,\allowbreak \frac{1}{\lambda},\allowbreak \frac{1}{\mu},\allowbreak \rho,\allowbreak q\Bigr)=(2.3,\allowbreak -3.5,\allowbreak 0.8,\allowbreak 1.1,\allowbreak 0.5,\allowbreak -0.7)$.}
\label{fig:_3}
\end{figure}

We start by considering \eqref{distributed delay} with parameters $\Bigl(\sigma^+,\allowbreak \sigma^-,\allowbreak \frac{1}{\lambda},\allowbreak \frac{1}{\mu},\allowbreak \rho,\allowbreak q\Bigr)=(2.3,\allowbreak -3.5,\allowbreak 0.8,\allowbreak 1.1,\allowbreak 0.5,\allowbreak -0.7)$. For this set of parameters, Figure~\ref{fig:_2} shows the function $N$ and its truncations $N_0$, $N_1$, $N_2$, $N_3$. We observe that these low-order approximations are already visually close to the graph of $N$, in particular $N_2$ and $N_3$. Solutions of \eqref{distributed delay} with initial condition $\beta(t, 1) =\sin (\pi t)$ for $t \in [-\tau, 0]$ for $N$ and for the truncations $N_0$, $N_1$, $N_2$, and $N_3$ are represented in Figure~\ref{fig:_3}, in which we observe that these solutions diverge for the truncations $N_0$, $N_1$, and $N_2$, but converge to $0$ for $N_3$ and $N$. Indeed, \eqref{nece_suf} is not satisfied for $N_0$, $N_1$ and $N_2$, since
\begin{align*}
\int_0^\tau N_0(\nu) \diff \nu & \approx 1.6561>1.35=1-q\rho, \\
\int_0^\tau N_1(\nu) \diff \nu & \approx 1.6117>1.35=1-q\rho, \\
\intertext{and}
\int_0^\tau N_2(\nu) \diff \nu & \approx 1.3768>1.35=1-q\rho.
\end{align*}

\begin{figure}[htp]
\centering
\includegraphics[width=0.7\textwidth]{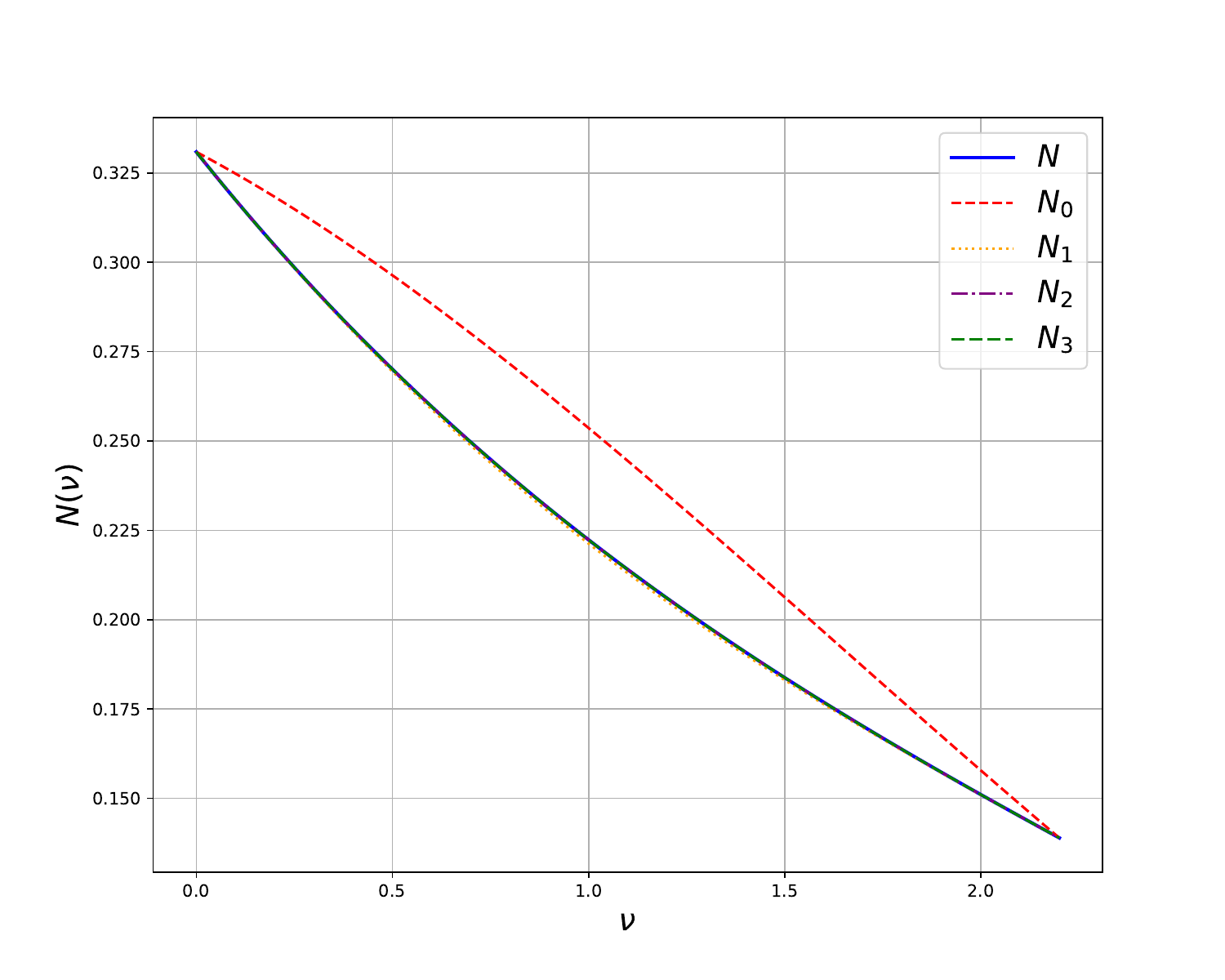}
\caption{$N$, $N_0$, $N_1$, $N_2$, and $N_3$ with $\Bigl(\sigma^+, \sigma^-, \frac{1}{\lambda}, \frac{1}{\mu}, \rho, q\Bigr)=(1.1,0.4,1,1.2,0.4,-0.5)$.}
\label{fig:_4}
\end{figure}

\begin{figure}[htp]
\centering
\includegraphics[width=0.7\textwidth]{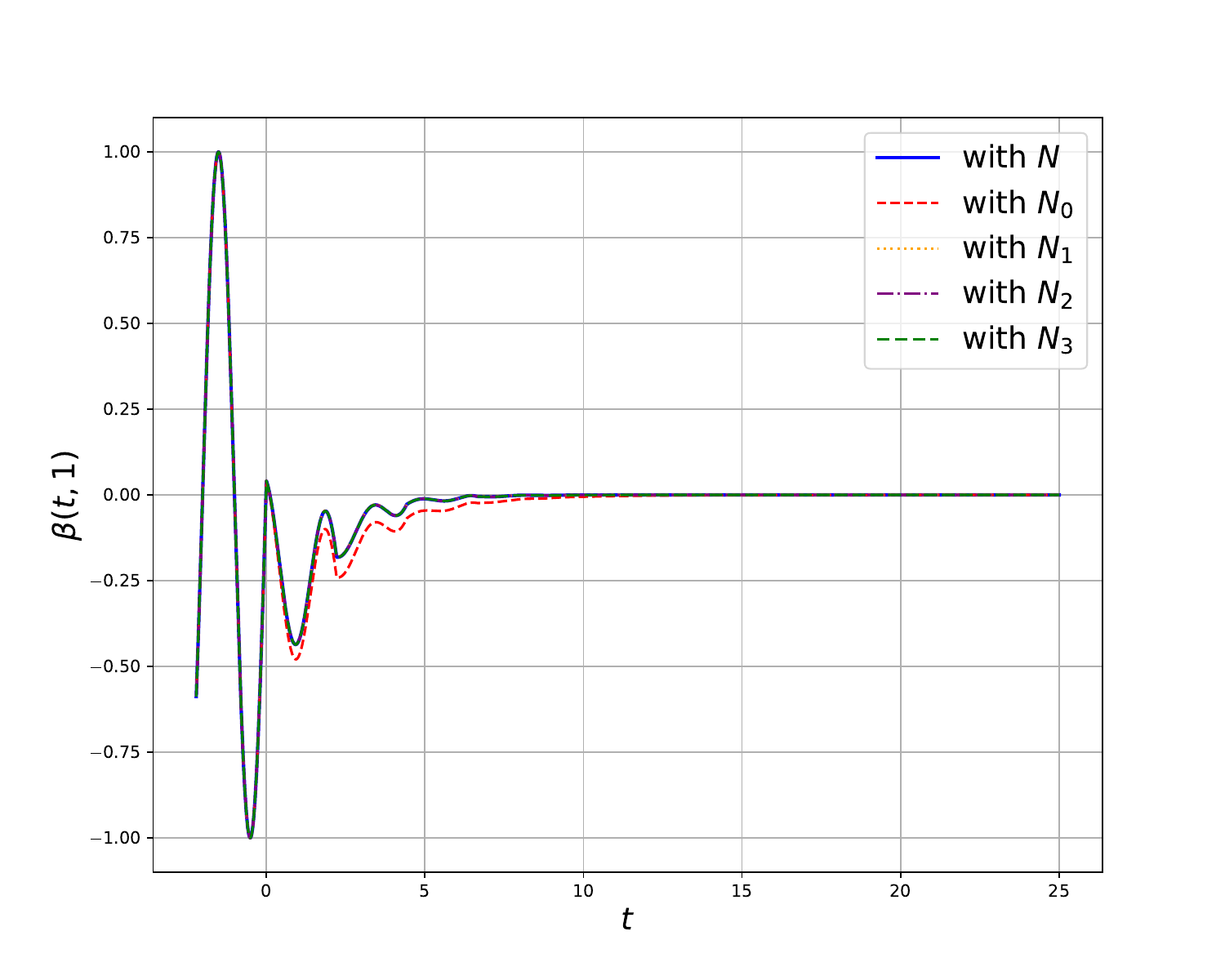}
\caption{Solution of \eqref{distributed delay} with $N$, $N_0$, $N_1$, $N_2$, and $N_3$ with $\Bigl(\sigma^+,\allowbreak \sigma^-,\allowbreak \frac{1}{\lambda},\allowbreak \frac{1}{\mu},\allowbreak \rho,\allowbreak q\Bigr)=(1.1,\allowbreak 0.4,\allowbreak 1,\allowbreak 1.2,\allowbreak 0.4,\allowbreak -0.5)$.}
\label{fig:_5}
\end{figure}

We also consider \eqref{distributed delay} with parameters $\Bigl(\sigma^+,\allowbreak \sigma^-,\allowbreak \frac{1}{\lambda},\allowbreak \frac{1}{\mu},\allowbreak \rho,\allowbreak q\Bigr)=(1.1,\allowbreak 0.4,\allowbreak 1,\allowbreak 1.2,\allowbreak 0.4,\allowbreak -0.5)$, representing the corresponding functions $N$, $N_0$, $N_1$, $N_2$, and $N_3$ in Figure~\ref{fig:_4} and the solutions of \eqref{distributed delay} with the same initial condition as before for these functions in Figure~\ref{fig:_5}. We observe that, for these parameters, a lower-order truncation already provides a good approximation for $N$, and all represented solutions converge, even for the lowest-order truncation $N_0$ of $N$. For this set of parameters, we also represent in Figure~\ref{fig:_6} the roots of the characteristic function $\Delta$ from \eqref{characteristic equation} for $N$, $N_0$, and $N_1$, which also confirm the stability observed in Figure~\ref{fig:_5}.

\begin{figure}[htp]
\centering
\includegraphics[width=0.7\textwidth]{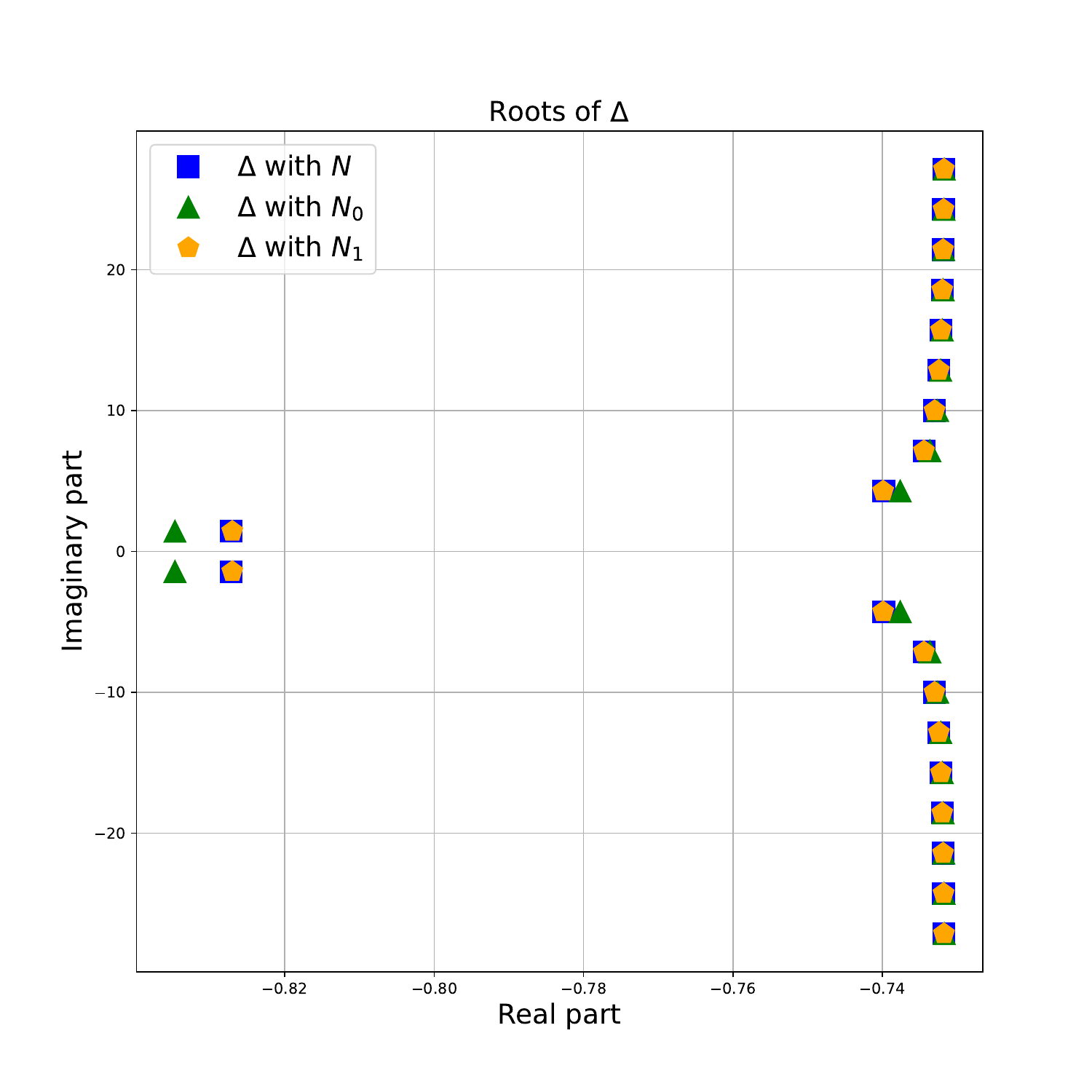}
\caption{Roots of $\Delta$ with with $N$, $N_0$, and $N_1$ with $\Bigl(\sigma^+, \sigma^-, \frac{1}{\lambda}, \frac{1}{\mu}, \rho, q\Bigr)=(1.1,\allowbreak 0.4,\allowbreak 1,\allowbreak 1.2,\allowbreak 0.4,\allowbreak -0.5)$.}
\label{fig:_6}
\end{figure}

We conclude this section with the following result, which provides a sufficient stability criterion for \eqref{eq:hyperbolic_couple} and \eqref{distributed delay} with constant parameters $\sigma^+$ and $\sigma^-$ (and hence, with $N$ given by \eqref{eq:explicit-N-Bessel}) in terms of properties of the function $\Delta_p$ from \eqref{eq:Delta-p} for some suitable $p$.

\begin{proposition}\label{proposition}
Assume that there exist $p_0\in \N$ and $\epsilon_0$ such that
\begin{equation}
\label{Del_del_p}
\left\{
\begin{aligned}
\inf_{\real(s)\geq 0} \lvert  \Delta_{p_0}(s)\rvert & >\epsilon_0, \\
\left\lVert N-N_{p_0}\right\rVert_{L^1(0, \tau)} & <\epsilon_0,
\end{aligned}
\right.
\end{equation}
where $\Delta_{p_0}$ is defined from $N_{p_0}$ as in \eqref{eq:Delta-p}. Then
\begin{equation}
\inf_{\real(s)\geq 0} \lvert \Delta(s)\rvert>0.
\end{equation}
\end{proposition}

\begin{proof}
For all $s\in\C$, we have
\begin{equation}
\Delta(s)= \Delta_{p_0}(s)-\int_0^\tau \left(N(\nu)-N_{p_0}(\nu)\right) \mathrm{e}^{-s \nu} \diff \nu.
\end{equation} Thus,  for all $s\in\C$ with $\real(s) \geq 0$, we obtain
\begin{align}
\lvert  \Delta(s)\rvert&\geq\lvert \Delta_{p_0}(s)\rvert-\left\lvert \int_0^\tau \left(N(\nu)-N_{p_0}(\nu)\right) \mathrm{e}^{-s \nu} \diff \nu\right\rvert\\& \geq\lvert \Delta_{p_0}(s)\rvert- \left\lVert N-N_{p_0}\right\rVert_{L^1(0,\tau)}.
\end{align} Then, from \eqref{Del_del_p}, we obtain
\begin{equation}
\inf_{\real(s)\geq 0} \lvert  \Delta(s)\rvert>0,
\end{equation}
as required.
\end{proof}

\section{Conclusion}
\label{sec_conclusion}

In this paper, we presented a novel necessary and sufficient condition for $2\times 2$ first-order linear hyperbolic systems. First, the problem was reformulated as a stability problem for an associated integral difference equation. Then, we provided a result counting the number of unstable roots of the latter system, yielding, as a consequence, a necessary and sufficient stability criterion for the system of first-order linear hyperbolic partial differential equations. In future work, we plan to study non-scalar systems and investigate the stability of integral difference equations with multiple delays.

\bibliographystyle{abbrv} 
\bibliography{biblio}

\end{document}

%% file: Figures/norm_L2_solution.pgf
\begingroup%
\makeatletter%
\begin{pgfpicture}%
\pgfpathrectangle{\pgfpointorigin}{\pgfqpoint{5.000000in}{5.000000in}}%
\pgfusepath{use as bounding box, clip}%
\begin{pgfscope}%
\pgfsetbuttcap%
\pgfsetmiterjoin%
\definecolor{currentfill}{rgb}{1.000000,1.000000,1.000000}%
\pgfsetfillcolor{currentfill}%
\pgfsetlinewidth{0.000000pt}%
\definecolor{currentstroke}{rgb}{1.000000,1.000000,1.000000}%
\pgfsetstrokecolor{currentstroke}%
\pgfsetdash{}{0pt}%
\pgfpathmoveto{\pgfqpoint{0.000000in}{0.000000in}}%
\pgfpathlineto{\pgfqpoint{5.000000in}{0.000000in}}%
\pgfpathlineto{\pgfqpoint{5.000000in}{5.000000in}}%
\pgfpathlineto{\pgfqpoint{0.000000in}{5.000000in}}%
\pgfpathlineto{\pgfqpoint{0.000000in}{0.000000in}}%
\pgfpathclose%
\pgfusepath{fill}%
\end{pgfscope}%
\begin{pgfscope}%
\pgfsetbuttcap%
\pgfsetmiterjoin%
\definecolor{currentfill}{rgb}{1.000000,1.000000,1.000000}%
\pgfsetfillcolor{currentfill}%
\pgfsetlinewidth{0.000000pt}%
\definecolor{currentstroke}{rgb}{0.000000,0.000000,0.000000}%
\pgfsetstrokecolor{currentstroke}%
\pgfsetstrokeopacity{0.000000}%
\pgfsetdash{}{0pt}%
\pgfpathmoveto{\pgfqpoint{0.625000in}{0.550000in}}%
\pgfpathlineto{\pgfqpoint{4.500000in}{0.550000in}}%
\pgfpathlineto{\pgfqpoint{4.500000in}{4.400000in}}%
\pgfpathlineto{\pgfqpoint{0.625000in}{4.400000in}}%
\pgfpathlineto{\pgfqpoint{0.625000in}{0.550000in}}%
\pgfpathclose%
\pgfusepath{fill}%
\end{pgfscope}%
\begin{pgfscope}%
\pgfpathrectangle{\pgfqpoint{0.625000in}{0.550000in}}{\pgfqpoint{3.875000in}{3.850000in}}%
\pgfusepath{clip}%
\pgfsetrectcap%
\pgfsetroundjoin%
\pgfsetlinewidth{0.803000pt}%
\definecolor{currentstroke}{rgb}{0.690196,0.690196,0.690196}%
\pgfsetstrokecolor{currentstroke}%
\pgfsetdash{}{0pt}%
\pgfpathmoveto{\pgfqpoint{0.801136in}{0.550000in}}%
\pgfpathlineto{\pgfqpoint{0.801136in}{4.400000in}}%
\pgfusepath{stroke}%
\end{pgfscope}%
\begin{pgfscope}%
\pgfsetbuttcap%
\pgfsetroundjoin%
\definecolor{currentfill}{rgb}{0.000000,0.000000,0.000000}%
\pgfsetfillcolor{currentfill}%
\pgfsetlinewidth{0.803000pt}%
\definecolor{currentstroke}{rgb}{0.000000,0.000000,0.000000}%
\pgfsetstrokecolor{currentstroke}%
\pgfsetdash{}{0pt}%
\pgfsys@defobject{currentmarker}{\pgfqpoint{0.000000in}{-0.048611in}}{\pgfqpoint{0.000000in}{0.000000in}}{%
\pgfpathmoveto{\pgfqpoint{0.000000in}{0.000000in}}%
\pgfpathlineto{\pgfqpoint{0.000000in}{-0.048611in}}%
\pgfusepath{stroke,fill}%
}%
\begin{pgfscope}%
\pgfsys@transformshift{0.801136in}{0.550000in}%
\pgfsys@useobject{currentmarker}{}%
\end{pgfscope}%
\end{pgfscope}%
\begin{pgfscope}%
\definecolor{textcolor}{rgb}{0.000000,0.000000,0.000000}%
\pgfsetstrokecolor{textcolor}%
\pgfsetfillcolor{textcolor}%
\pgftext[x=0.801136in,y=0.452778in,,top]{\color{textcolor}{\sffamily\fontsize{10.000000}{12.000000}\selectfont\catcode`\^=\active\def^{\ifmmode\sp\else\^{}\fi}\catcode`\%=\active\def
\end{pgfscope}%
\begin{pgfscope}%
\pgfpathrectangle{\pgfqpoint{0.625000in}{0.550000in}}{\pgfqpoint{3.875000in}{3.850000in}}%
\pgfusepath{clip}%
\pgfsetrectcap%
\pgfsetroundjoin%
\pgfsetlinewidth{0.803000pt}%
\definecolor{currentstroke}{rgb}{0.690196,0.690196,0.690196}%
\pgfsetstrokecolor{currentstroke}%
\pgfsetdash{}{0pt}%
\pgfpathmoveto{\pgfqpoint{1.505682in}{0.550000in}}%
\pgfpathlineto{\pgfqpoint{1.505682in}{4.400000in}}%
\pgfusepath{stroke}%
\end{pgfscope}%
\begin{pgfscope}%
\pgfsetbuttcap%
\pgfsetroundjoin%
\definecolor{currentfill}{rgb}{0.000000,0.000000,0.000000}%
\pgfsetfillcolor{currentfill}%
\pgfsetlinewidth{0.803000pt}%
\definecolor{currentstroke}{rgb}{0.000000,0.000000,0.000000}%
\pgfsetstrokecolor{currentstroke}%
\pgfsetdash{}{0pt}%
\pgfsys@defobject{currentmarker}{\pgfqpoint{0.000000in}{-0.048611in}}{\pgfqpoint{0.000000in}{0.000000in}}{%
\pgfpathmoveto{\pgfqpoint{0.000000in}{0.000000in}}%
\pgfpathlineto{\pgfqpoint{0.000000in}{-0.048611in}}%
\pgfusepath{stroke,fill}%
}%
\begin{pgfscope}%
\pgfsys@transformshift{1.505682in}{0.550000in}%
\pgfsys@useobject{currentmarker}{}%
\end{pgfscope}%
\end{pgfscope}%
\begin{pgfscope}%
\definecolor{textcolor}{rgb}{0.000000,0.000000,0.000000}%
\pgfsetstrokecolor{textcolor}%
\pgfsetfillcolor{textcolor}%
\pgftext[x=1.505682in,y=0.452778in,,top]{\color{textcolor}{\sffamily\fontsize{10.000000}{12.000000}\selectfont\catcode`\^=\active\def^{\ifmmode\sp\else\^{}\fi}\catcode`\%=\active\def
\end{pgfscope}%
\begin{pgfscope}%
\pgfpathrectangle{\pgfqpoint{0.625000in}{0.550000in}}{\pgfqpoint{3.875000in}{3.850000in}}%
\pgfusepath{clip}%
\pgfsetrectcap%
\pgfsetroundjoin%
\pgfsetlinewidth{0.803000pt}%
\definecolor{currentstroke}{rgb}{0.690196,0.690196,0.690196}%
\pgfsetstrokecolor{currentstroke}%
\pgfsetdash{}{0pt}%
\pgfpathmoveto{\pgfqpoint{2.210227in}{0.550000in}}%
\pgfpathlineto{\pgfqpoint{2.210227in}{4.400000in}}%
\pgfusepath{stroke}%
\end{pgfscope}%
\begin{pgfscope}%
\pgfsetbuttcap%
\pgfsetroundjoin%
\definecolor{currentfill}{rgb}{0.000000,0.000000,0.000000}%
\pgfsetfillcolor{currentfill}%
\pgfsetlinewidth{0.803000pt}%
\definecolor{currentstroke}{rgb}{0.000000,0.000000,0.000000}%
\pgfsetstrokecolor{currentstroke}%
\pgfsetdash{}{0pt}%
\pgfsys@defobject{currentmarker}{\pgfqpoint{0.000000in}{-0.048611in}}{\pgfqpoint{0.000000in}{0.000000in}}{%
\pgfpathmoveto{\pgfqpoint{0.000000in}{0.000000in}}%
\pgfpathlineto{\pgfqpoint{0.000000in}{-0.048611in}}%
\pgfusepath{stroke,fill}%
}%
\begin{pgfscope}%
\pgfsys@transformshift{2.210227in}{0.550000in}%
\pgfsys@useobject{currentmarker}{}%
\end{pgfscope}%
\end{pgfscope}%
\begin{pgfscope}%
\definecolor{textcolor}{rgb}{0.000000,0.000000,0.000000}%
\pgfsetstrokecolor{textcolor}%
\pgfsetfillcolor{textcolor}%
\pgftext[x=2.210227in,y=0.452778in,,top]{\color{textcolor}{\sffamily\fontsize{10.000000}{12.000000}\selectfont\catcode`\^=\active\def^{\ifmmode\sp\else\^{}\fi}\catcode`\%=\active\def
\end{pgfscope}%
\begin{pgfscope}%
\pgfpathrectangle{\pgfqpoint{0.625000in}{0.550000in}}{\pgfqpoint{3.875000in}{3.850000in}}%
\pgfusepath{clip}%
\pgfsetrectcap%
\pgfsetroundjoin%
\pgfsetlinewidth{0.803000pt}%
\definecolor{currentstroke}{rgb}{0.690196,0.690196,0.690196}%
\pgfsetstrokecolor{currentstroke}%
\pgfsetdash{}{0pt}%
\pgfpathmoveto{\pgfqpoint{2.914773in}{0.550000in}}%
\pgfpathlineto{\pgfqpoint{2.914773in}{4.400000in}}%
\pgfusepath{stroke}%
\end{pgfscope}%
\begin{pgfscope}%
\pgfsetbuttcap%
\pgfsetroundjoin%
\definecolor{currentfill}{rgb}{0.000000,0.000000,0.000000}%
\pgfsetfillcolor{currentfill}%
\pgfsetlinewidth{0.803000pt}%
\definecolor{currentstroke}{rgb}{0.000000,0.000000,0.000000}%
\pgfsetstrokecolor{currentstroke}%
\pgfsetdash{}{0pt}%
\pgfsys@defobject{currentmarker}{\pgfqpoint{0.000000in}{-0.048611in}}{\pgfqpoint{0.000000in}{0.000000in}}{%
\pgfpathmoveto{\pgfqpoint{0.000000in}{0.000000in}}%
\pgfpathlineto{\pgfqpoint{0.000000in}{-0.048611in}}%
\pgfusepath{stroke,fill}%
}%
\begin{pgfscope}%
\pgfsys@transformshift{2.914773in}{0.550000in}%
\pgfsys@useobject{currentmarker}{}%
\end{pgfscope}%
\end{pgfscope}%
\begin{pgfscope}%
\definecolor{textcolor}{rgb}{0.000000,0.000000,0.000000}%
\pgfsetstrokecolor{textcolor}%
\pgfsetfillcolor{textcolor}%
\pgftext[x=2.914773in,y=0.452778in,,top]{\color{textcolor}{\sffamily\fontsize{10.000000}{12.000000}\selectfont\catcode`\^=\active\def^{\ifmmode\sp\else\^{}\fi}\catcode`\%=\active\def
\end{pgfscope}%
\begin{pgfscope}%
\pgfpathrectangle{\pgfqpoint{0.625000in}{0.550000in}}{\pgfqpoint{3.875000in}{3.850000in}}%
\pgfusepath{clip}%
\pgfsetrectcap%
\pgfsetroundjoin%
\pgfsetlinewidth{0.803000pt}%
\definecolor{currentstroke}{rgb}{0.690196,0.690196,0.690196}%
\pgfsetstrokecolor{currentstroke}%
\pgfsetdash{}{0pt}%
\pgfpathmoveto{\pgfqpoint{3.619318in}{0.550000in}}%
\pgfpathlineto{\pgfqpoint{3.619318in}{4.400000in}}%
\pgfusepath{stroke}%
\end{pgfscope}%
\begin{pgfscope}%
\pgfsetbuttcap%
\pgfsetroundjoin%
\definecolor{currentfill}{rgb}{0.000000,0.000000,0.000000}%
\pgfsetfillcolor{currentfill}%
\pgfsetlinewidth{0.803000pt}%
\definecolor{currentstroke}{rgb}{0.000000,0.000000,0.000000}%
\pgfsetstrokecolor{currentstroke}%
\pgfsetdash{}{0pt}%
\pgfsys@defobject{currentmarker}{\pgfqpoint{0.000000in}{-0.048611in}}{\pgfqpoint{0.000000in}{0.000000in}}{%
\pgfpathmoveto{\pgfqpoint{0.000000in}{0.000000in}}%
\pgfpathlineto{\pgfqpoint{0.000000in}{-0.048611in}}%
\pgfusepath{stroke,fill}%
}%
\begin{pgfscope}%
\pgfsys@transformshift{3.619318in}{0.550000in}%
\pgfsys@useobject{currentmarker}{}%
\end{pgfscope}%
\end{pgfscope}%
\begin{pgfscope}%
\definecolor{textcolor}{rgb}{0.000000,0.000000,0.000000}%
\pgfsetstrokecolor{textcolor}%
\pgfsetfillcolor{textcolor}%
\pgftext[x=3.619318in,y=0.452778in,,top]{\color{textcolor}{\sffamily\fontsize{10.000000}{12.000000}\selectfont\catcode`\^=\active\def^{\ifmmode\sp\else\^{}\fi}\catcode`\%=\active\def
\end{pgfscope}%
\begin{pgfscope}%
\pgfpathrectangle{\pgfqpoint{0.625000in}{0.550000in}}{\pgfqpoint{3.875000in}{3.850000in}}%
\pgfusepath{clip}%
\pgfsetrectcap%
\pgfsetroundjoin%
\pgfsetlinewidth{0.803000pt}%
\definecolor{currentstroke}{rgb}{0.690196,0.690196,0.690196}%
\pgfsetstrokecolor{currentstroke}%
\pgfsetdash{}{0pt}%
\pgfpathmoveto{\pgfqpoint{4.323864in}{0.550000in}}%
\pgfpathlineto{\pgfqpoint{4.323864in}{4.400000in}}%
\pgfusepath{stroke}%
\end{pgfscope}%
\begin{pgfscope}%
\pgfsetbuttcap%
\pgfsetroundjoin%
\definecolor{currentfill}{rgb}{0.000000,0.000000,0.000000}%
\pgfsetfillcolor{currentfill}%
\pgfsetlinewidth{0.803000pt}%
\definecolor{currentstroke}{rgb}{0.000000,0.000000,0.000000}%
\pgfsetstrokecolor{currentstroke}%
\pgfsetdash{}{0pt}%
\pgfsys@defobject{currentmarker}{\pgfqpoint{0.000000in}{-0.048611in}}{\pgfqpoint{0.000000in}{0.000000in}}{%
\pgfpathmoveto{\pgfqpoint{0.000000in}{0.000000in}}%
\pgfpathlineto{\pgfqpoint{0.000000in}{-0.048611in}}%
\pgfusepath{stroke,fill}%
}%
\begin{pgfscope}%
\pgfsys@transformshift{4.323864in}{0.550000in}%
\pgfsys@useobject{currentmarker}{}%
\end{pgfscope}%
\end{pgfscope}%
\begin{pgfscope}%
\definecolor{textcolor}{rgb}{0.000000,0.000000,0.000000}%
\pgfsetstrokecolor{textcolor}%
\pgfsetfillcolor{textcolor}%
\pgftext[x=4.323864in,y=0.452778in,,top]{\color{textcolor}{\sffamily\fontsize{10.000000}{12.000000}\selectfont\catcode`\^=\active\def^{\ifmmode\sp\else\^{}\fi}\catcode`\%=\active\def
\end{pgfscope}%
\begin{pgfscope}%
\definecolor{textcolor}{rgb}{0.000000,0.000000,0.000000}%
\pgfsetstrokecolor{textcolor}%
\pgfsetfillcolor{textcolor}%
\pgftext[x=2.562500in,y=0.262809in,,top]{\color{textcolor}{\sffamily\fontsize{10.000000}{12.000000}\selectfont\catcode`\^=\active\def^{\ifmmode\sp\else\^{}\fi}\catcode`\%=\active\def
\end{pgfscope}%
\begin{pgfscope}%
\pgfpathrectangle{\pgfqpoint{0.625000in}{0.550000in}}{\pgfqpoint{3.875000in}{3.850000in}}%
\pgfusepath{clip}%
\pgfsetrectcap%
\pgfsetroundjoin%
\pgfsetlinewidth{0.803000pt}%
\definecolor{currentstroke}{rgb}{0.690196,0.690196,0.690196}%
\pgfsetstrokecolor{currentstroke}%
\pgfsetdash{}{0pt}%
\pgfpathmoveto{\pgfqpoint{0.625000in}{0.724935in}}%
\pgfpathlineto{\pgfqpoint{4.500000in}{0.724935in}}%
\pgfusepath{stroke}%
\end{pgfscope}%
\begin{pgfscope}%
\pgfsetbuttcap%
\pgfsetroundjoin%
\definecolor{currentfill}{rgb}{0.000000,0.000000,0.000000}%
\pgfsetfillcolor{currentfill}%
\pgfsetlinewidth{0.803000pt}%
\definecolor{currentstroke}{rgb}{0.000000,0.000000,0.000000}%
\pgfsetstrokecolor{currentstroke}%
\pgfsetdash{}{0pt}%
\pgfsys@defobject{currentmarker}{\pgfqpoint{-0.048611in}{0.000000in}}{\pgfqpoint{-0.000000in}{0.000000in}}{%
\pgfpathmoveto{\pgfqpoint{-0.000000in}{0.000000in}}%
\pgfpathlineto{\pgfqpoint{-0.048611in}{0.000000in}}%
\pgfusepath{stroke,fill}%
}%
\begin{pgfscope}%
\pgfsys@transformshift{0.625000in}{0.724935in}%
\pgfsys@useobject{currentmarker}{}%
\end{pgfscope}%
\end{pgfscope}%
\begin{pgfscope}%
\definecolor{textcolor}{rgb}{0.000000,0.000000,0.000000}%
\pgfsetstrokecolor{textcolor}%
\pgfsetfillcolor{textcolor}%
\pgftext[x=0.306898in, y=0.672173in, left, base]{\color{textcolor}{\sffamily\fontsize{10.000000}{12.000000}\selectfont\catcode`\^=\active\def^{\ifmmode\sp\else\^{}\fi}\catcode`\%=\active\def
\end{pgfscope}%
\begin{pgfscope}%
\pgfpathrectangle{\pgfqpoint{0.625000in}{0.550000in}}{\pgfqpoint{3.875000in}{3.850000in}}%
\pgfusepath{clip}%
\pgfsetrectcap%
\pgfsetroundjoin%
\pgfsetlinewidth{0.803000pt}%
\definecolor{currentstroke}{rgb}{0.690196,0.690196,0.690196}%
\pgfsetstrokecolor{currentstroke}%
\pgfsetdash{}{0pt}%
\pgfpathmoveto{\pgfqpoint{0.625000in}{1.291651in}}%
\pgfpathlineto{\pgfqpoint{4.500000in}{1.291651in}}%
\pgfusepath{stroke}%
\end{pgfscope}%
\begin{pgfscope}%
\pgfsetbuttcap%
\pgfsetroundjoin%
\definecolor{currentfill}{rgb}{0.000000,0.000000,0.000000}%
\pgfsetfillcolor{currentfill}%
\pgfsetlinewidth{0.803000pt}%
\definecolor{currentstroke}{rgb}{0.000000,0.000000,0.000000}%
\pgfsetstrokecolor{currentstroke}%
\pgfsetdash{}{0pt}%
\pgfsys@defobject{currentmarker}{\pgfqpoint{-0.048611in}{0.000000in}}{\pgfqpoint{-0.000000in}{0.000000in}}{%
\pgfpathmoveto{\pgfqpoint{-0.000000in}{0.000000in}}%
\pgfpathlineto{\pgfqpoint{-0.048611in}{0.000000in}}%
\pgfusepath{stroke,fill}%
}%
\begin{pgfscope}%
\pgfsys@transformshift{0.625000in}{1.291651in}%
\pgfsys@useobject{currentmarker}{}%
\end{pgfscope}%
\end{pgfscope}%
\begin{pgfscope}%
\definecolor{textcolor}{rgb}{0.000000,0.000000,0.000000}%
\pgfsetstrokecolor{textcolor}%
\pgfsetfillcolor{textcolor}%
\pgftext[x=0.306898in, y=1.238890in, left, base]{\color{textcolor}{\sffamily\fontsize{10.000000}{12.000000}\selectfont\catcode`\^=\active\def^{\ifmmode\sp\else\^{}\fi}\catcode`\%=\active\def
\end{pgfscope}%
\begin{pgfscope}%
\pgfpathrectangle{\pgfqpoint{0.625000in}{0.550000in}}{\pgfqpoint{3.875000in}{3.850000in}}%
\pgfusepath{clip}%
\pgfsetrectcap%
\pgfsetroundjoin%
\pgfsetlinewidth{0.803000pt}%
\definecolor{currentstroke}{rgb}{0.690196,0.690196,0.690196}%
\pgfsetstrokecolor{currentstroke}%
\pgfsetdash{}{0pt}%
\pgfpathmoveto{\pgfqpoint{0.625000in}{1.858368in}}%
\pgfpathlineto{\pgfqpoint{4.500000in}{1.858368in}}%
\pgfusepath{stroke}%
\end{pgfscope}%
\begin{pgfscope}%
\pgfsetbuttcap%
\pgfsetroundjoin%
\definecolor{currentfill}{rgb}{0.000000,0.000000,0.000000}%
\pgfsetfillcolor{currentfill}%
\pgfsetlinewidth{0.803000pt}%
\definecolor{currentstroke}{rgb}{0.000000,0.000000,0.000000}%
\pgfsetstrokecolor{currentstroke}%
\pgfsetdash{}{0pt}%
\pgfsys@defobject{currentmarker}{\pgfqpoint{-0.048611in}{0.000000in}}{\pgfqpoint{-0.000000in}{0.000000in}}{%
\pgfpathmoveto{\pgfqpoint{-0.000000in}{0.000000in}}%
\pgfpathlineto{\pgfqpoint{-0.048611in}{0.000000in}}%
\pgfusepath{stroke,fill}%
}%
\begin{pgfscope}%
\pgfsys@transformshift{0.625000in}{1.858368in}%
\pgfsys@useobject{currentmarker}{}%
\end{pgfscope}%
\end{pgfscope}%
\begin{pgfscope}%
\definecolor{textcolor}{rgb}{0.000000,0.000000,0.000000}%
\pgfsetstrokecolor{textcolor}%
\pgfsetfillcolor{textcolor}%
\pgftext[x=0.306898in, y=1.805606in, left, base]{\color{textcolor}{\sffamily\fontsize{10.000000}{12.000000}\selectfont\catcode`\^=\active\def^{\ifmmode\sp\else\^{}\fi}\catcode`\%=\active\def
\end{pgfscope}%
\begin{pgfscope}%
\pgfpathrectangle{\pgfqpoint{0.625000in}{0.550000in}}{\pgfqpoint{3.875000in}{3.850000in}}%
\pgfusepath{clip}%
\pgfsetrectcap%
\pgfsetroundjoin%
\pgfsetlinewidth{0.803000pt}%
\definecolor{currentstroke}{rgb}{0.690196,0.690196,0.690196}%
\pgfsetstrokecolor{currentstroke}%
\pgfsetdash{}{0pt}%
\pgfpathmoveto{\pgfqpoint{0.625000in}{2.425085in}}%
\pgfpathlineto{\pgfqpoint{4.500000in}{2.425085in}}%
\pgfusepath{stroke}%
\end{pgfscope}%
\begin{pgfscope}%
\pgfsetbuttcap%
\pgfsetroundjoin%
\definecolor{currentfill}{rgb}{0.000000,0.000000,0.000000}%
\pgfsetfillcolor{currentfill}%
\pgfsetlinewidth{0.803000pt}%
\definecolor{currentstroke}{rgb}{0.000000,0.000000,0.000000}%
\pgfsetstrokecolor{currentstroke}%
\pgfsetdash{}{0pt}%
\pgfsys@defobject{currentmarker}{\pgfqpoint{-0.048611in}{0.000000in}}{\pgfqpoint{-0.000000in}{0.000000in}}{%
\pgfpathmoveto{\pgfqpoint{-0.000000in}{0.000000in}}%
\pgfpathlineto{\pgfqpoint{-0.048611in}{0.000000in}}%
\pgfusepath{stroke,fill}%
}%
\begin{pgfscope}%
\pgfsys@transformshift{0.625000in}{2.425085in}%
\pgfsys@useobject{currentmarker}{}%
\end{pgfscope}%
\end{pgfscope}%
\begin{pgfscope}%
\definecolor{textcolor}{rgb}{0.000000,0.000000,0.000000}%
\pgfsetstrokecolor{textcolor}%
\pgfsetfillcolor{textcolor}%
\pgftext[x=0.306898in, y=2.372323in, left, base]{\color{textcolor}{\sffamily\fontsize{10.000000}{12.000000}\selectfont\catcode`\^=\active\def^{\ifmmode\sp\else\^{}\fi}\catcode`\%=\active\def
\end{pgfscope}%
\begin{pgfscope}%
\pgfpathrectangle{\pgfqpoint{0.625000in}{0.550000in}}{\pgfqpoint{3.875000in}{3.850000in}}%
\pgfusepath{clip}%
\pgfsetrectcap%
\pgfsetroundjoin%
\pgfsetlinewidth{0.803000pt}%
\definecolor{currentstroke}{rgb}{0.690196,0.690196,0.690196}%
\pgfsetstrokecolor{currentstroke}%
\pgfsetdash{}{0pt}%
\pgfpathmoveto{\pgfqpoint{0.625000in}{2.991801in}}%
\pgfpathlineto{\pgfqpoint{4.500000in}{2.991801in}}%
\pgfusepath{stroke}%
\end{pgfscope}%
\begin{pgfscope}%
\pgfsetbuttcap%
\pgfsetroundjoin%
\definecolor{currentfill}{rgb}{0.000000,0.000000,0.000000}%
\pgfsetfillcolor{currentfill}%
\pgfsetlinewidth{0.803000pt}%
\definecolor{currentstroke}{rgb}{0.000000,0.000000,0.000000}%
\pgfsetstrokecolor{currentstroke}%
\pgfsetdash{}{0pt}%
\pgfsys@defobject{currentmarker}{\pgfqpoint{-0.048611in}{0.000000in}}{\pgfqpoint{-0.000000in}{0.000000in}}{%
\pgfpathmoveto{\pgfqpoint{-0.000000in}{0.000000in}}%
\pgfpathlineto{\pgfqpoint{-0.048611in}{0.000000in}}%
\pgfusepath{stroke,fill}%
}%
\begin{pgfscope}%
\pgfsys@transformshift{0.625000in}{2.991801in}%
\pgfsys@useobject{currentmarker}{}%
\end{pgfscope}%
\end{pgfscope}%
\begin{pgfscope}%
\definecolor{textcolor}{rgb}{0.000000,0.000000,0.000000}%
\pgfsetstrokecolor{textcolor}%
\pgfsetfillcolor{textcolor}%
\pgftext[x=0.306898in, y=2.939040in, left, base]{\color{textcolor}{\sffamily\fontsize{10.000000}{12.000000}\selectfont\catcode`\^=\active\def^{\ifmmode\sp\else\^{}\fi}\catcode`\%=\active\def
\end{pgfscope}%
\begin{pgfscope}%
\pgfpathrectangle{\pgfqpoint{0.625000in}{0.550000in}}{\pgfqpoint{3.875000in}{3.850000in}}%
\pgfusepath{clip}%
\pgfsetrectcap%
\pgfsetroundjoin%
\pgfsetlinewidth{0.803000pt}%
\definecolor{currentstroke}{rgb}{0.690196,0.690196,0.690196}%
\pgfsetstrokecolor{currentstroke}%
\pgfsetdash{}{0pt}%
\pgfpathmoveto{\pgfqpoint{0.625000in}{3.558518in}}%
\pgfpathlineto{\pgfqpoint{4.500000in}{3.558518in}}%
\pgfusepath{stroke}%
\end{pgfscope}%
\begin{pgfscope}%
\pgfsetbuttcap%
\pgfsetroundjoin%
\definecolor{currentfill}{rgb}{0.000000,0.000000,0.000000}%
\pgfsetfillcolor{currentfill}%
\pgfsetlinewidth{0.803000pt}%
\definecolor{currentstroke}{rgb}{0.000000,0.000000,0.000000}%
\pgfsetstrokecolor{currentstroke}%
\pgfsetdash{}{0pt}%
\pgfsys@defobject{currentmarker}{\pgfqpoint{-0.048611in}{0.000000in}}{\pgfqpoint{-0.000000in}{0.000000in}}{%
\pgfpathmoveto{\pgfqpoint{-0.000000in}{0.000000in}}%
\pgfpathlineto{\pgfqpoint{-0.048611in}{0.000000in}}%
\pgfusepath{stroke,fill}%
}%
\begin{pgfscope}%
\pgfsys@transformshift{0.625000in}{3.558518in}%
\pgfsys@useobject{currentmarker}{}%
\end{pgfscope}%
\end{pgfscope}%
\begin{pgfscope}%
\definecolor{textcolor}{rgb}{0.000000,0.000000,0.000000}%
\pgfsetstrokecolor{textcolor}%
\pgfsetfillcolor{textcolor}%
\pgftext[x=0.306898in, y=3.505756in, left, base]{\color{textcolor}{\sffamily\fontsize{10.000000}{12.000000}\selectfont\catcode`\^=\active\def^{\ifmmode\sp\else\^{}\fi}\catcode`\%=\active\def
\end{pgfscope}%
\begin{pgfscope}%
\pgfpathrectangle{\pgfqpoint{0.625000in}{0.550000in}}{\pgfqpoint{3.875000in}{3.850000in}}%
\pgfusepath{clip}%
\pgfsetrectcap%
\pgfsetroundjoin%
\pgfsetlinewidth{0.803000pt}%
\definecolor{currentstroke}{rgb}{0.690196,0.690196,0.690196}%
\pgfsetstrokecolor{currentstroke}%
\pgfsetdash{}{0pt}%
\pgfpathmoveto{\pgfqpoint{0.625000in}{4.125235in}}%
\pgfpathlineto{\pgfqpoint{4.500000in}{4.125235in}}%
\pgfusepath{stroke}%
\end{pgfscope}%
\begin{pgfscope}%
\pgfsetbuttcap%
\pgfsetroundjoin%
\definecolor{currentfill}{rgb}{0.000000,0.000000,0.000000}%
\pgfsetfillcolor{currentfill}%
\pgfsetlinewidth{0.803000pt}%
\definecolor{currentstroke}{rgb}{0.000000,0.000000,0.000000}%
\pgfsetstrokecolor{currentstroke}%
\pgfsetdash{}{0pt}%
\pgfsys@defobject{currentmarker}{\pgfqpoint{-0.048611in}{0.000000in}}{\pgfqpoint{-0.000000in}{0.000000in}}{%
\pgfpathmoveto{\pgfqpoint{-0.000000in}{0.000000in}}%
\pgfpathlineto{\pgfqpoint{-0.048611in}{0.000000in}}%
\pgfusepath{stroke,fill}%
}%
\begin{pgfscope}%
\pgfsys@transformshift{0.625000in}{4.125235in}%
\pgfsys@useobject{currentmarker}{}%
\end{pgfscope}%
\end{pgfscope}%
\begin{pgfscope}%
\definecolor{textcolor}{rgb}{0.000000,0.000000,0.000000}%
\pgfsetstrokecolor{textcolor}%
\pgfsetfillcolor{textcolor}%
\pgftext[x=0.306898in, y=4.072473in, left, base]{\color{textcolor}{\sffamily\fontsize{10.000000}{12.000000}\selectfont\catcode`\^=\active\def^{\ifmmode\sp\else\^{}\fi}\catcode`\%=\active\def
\end{pgfscope}%
\begin{pgfscope}%
\definecolor{textcolor}{rgb}{0.000000,0.000000,0.000000}%
\pgfsetstrokecolor{textcolor}%
\pgfsetfillcolor{textcolor}%
\pgftext[x=0.251343in,y=2.475000in,,bottom,rotate=90.000000]{\color{textcolor}{\sffamily\fontsize{10.000000}{12.000000}\selectfont\catcode`\^=\active\def^{\ifmmode\sp\else\^{}\fi}\catcode`\%=\active\def
\end{pgfscope}%
\begin{pgfscope}%
\pgfpathrectangle{\pgfqpoint{0.625000in}{0.550000in}}{\pgfqpoint{3.875000in}{3.850000in}}%
\pgfusepath{clip}%
\pgfsetrectcap%
\pgfsetroundjoin%
\pgfsetlinewidth{1.505625pt}%
\definecolor{currentstroke}{rgb}{0.121569,0.466667,0.705882}%
\pgfsetstrokecolor{currentstroke}%
\pgfsetdash{}{0pt}%
\pgfpathmoveto{\pgfqpoint{0.801136in}{4.225000in}}%
\pgfpathlineto{\pgfqpoint{0.803250in}{3.928197in}}%
\pgfpathlineto{\pgfqpoint{0.805364in}{3.663253in}}%
\pgfpathlineto{\pgfqpoint{0.809591in}{3.551858in}}%
\pgfpathlineto{\pgfqpoint{0.810295in}{3.558651in}}%
\pgfpathlineto{\pgfqpoint{0.812409in}{3.603041in}}%
\pgfpathlineto{\pgfqpoint{0.813114in}{3.593843in}}%
\pgfpathlineto{\pgfqpoint{0.816636in}{3.468355in}}%
\pgfpathlineto{\pgfqpoint{0.817341in}{3.470379in}}%
\pgfpathlineto{\pgfqpoint{0.819455in}{3.485009in}}%
\pgfpathlineto{\pgfqpoint{0.826500in}{3.376568in}}%
\pgfpathlineto{\pgfqpoint{0.827205in}{3.378649in}}%
\pgfpathlineto{\pgfqpoint{0.828614in}{3.396790in}}%
\pgfpathlineto{\pgfqpoint{0.829318in}{3.394232in}}%
\pgfpathlineto{\pgfqpoint{0.835659in}{3.313012in}}%
\pgfpathlineto{\pgfqpoint{0.837773in}{3.306464in}}%
\pgfpathlineto{\pgfqpoint{0.839182in}{3.300506in}}%
\pgfpathlineto{\pgfqpoint{0.841295in}{3.271311in}}%
\pgfpathlineto{\pgfqpoint{0.843409in}{3.249539in}}%
\pgfpathlineto{\pgfqpoint{0.846227in}{3.239324in}}%
\pgfpathlineto{\pgfqpoint{0.863841in}{3.104108in}}%
\pgfpathlineto{\pgfqpoint{0.868068in}{3.071909in}}%
\pgfpathlineto{\pgfqpoint{0.873000in}{3.039192in}}%
\pgfpathlineto{\pgfqpoint{0.896955in}{2.875509in}}%
\pgfpathlineto{\pgfqpoint{0.943455in}{2.592693in}}%
\pgfpathlineto{\pgfqpoint{0.971636in}{2.439568in}}%
\pgfpathlineto{\pgfqpoint{0.999114in}{2.302378in}}%
\pgfpathlineto{\pgfqpoint{1.026591in}{2.176168in}}%
\pgfpathlineto{\pgfqpoint{1.054068in}{2.060056in}}%
\pgfpathlineto{\pgfqpoint{1.082250in}{1.950610in}}%
\pgfpathlineto{\pgfqpoint{1.110432in}{1.850136in}}%
\pgfpathlineto{\pgfqpoint{1.138614in}{1.757899in}}%
\pgfpathlineto{\pgfqpoint{1.166795in}{1.673222in}}%
\pgfpathlineto{\pgfqpoint{1.194977in}{1.595487in}}%
\pgfpathlineto{\pgfqpoint{1.223159in}{1.524124in}}%
\pgfpathlineto{\pgfqpoint{1.252045in}{1.457044in}}%
\pgfpathlineto{\pgfqpoint{1.280932in}{1.395594in}}%
\pgfpathlineto{\pgfqpoint{1.309818in}{1.339302in}}%
\pgfpathlineto{\pgfqpoint{1.339409in}{1.286533in}}%
\pgfpathlineto{\pgfqpoint{1.369000in}{1.238296in}}%
\pgfpathlineto{\pgfqpoint{1.399295in}{1.193201in}}%
\pgfpathlineto{\pgfqpoint{1.429591in}{1.152066in}}%
\pgfpathlineto{\pgfqpoint{1.460591in}{1.113713in}}%
\pgfpathlineto{\pgfqpoint{1.492295in}{1.078047in}}%
\pgfpathlineto{\pgfqpoint{1.524705in}{1.044969in}}%
\pgfpathlineto{\pgfqpoint{1.557818in}{1.014370in}}%
\pgfpathlineto{\pgfqpoint{1.591636in}{0.986137in}}%
\pgfpathlineto{\pgfqpoint{1.626864in}{0.959652in}}%
\pgfpathlineto{\pgfqpoint{1.662795in}{0.935402in}}%
\pgfpathlineto{\pgfqpoint{1.700136in}{0.912852in}}%
\pgfpathlineto{\pgfqpoint{1.739591in}{0.891646in}}%
\pgfpathlineto{\pgfqpoint{1.780455in}{0.872201in}}%
\pgfpathlineto{\pgfqpoint{1.823432in}{0.854193in}}%
\pgfpathlineto{\pgfqpoint{1.869227in}{0.837420in}}%
\pgfpathlineto{\pgfqpoint{1.917841in}{0.821990in}}%
\pgfpathlineto{\pgfqpoint{1.969977in}{0.807787in}}%
\pgfpathlineto{\pgfqpoint{2.026341in}{0.794760in}}%
\pgfpathlineto{\pgfqpoint{2.087636in}{0.782907in}}%
\pgfpathlineto{\pgfqpoint{2.154568in}{0.772250in}}%
\pgfpathlineto{\pgfqpoint{2.229250in}{0.762654in}}%
\pgfpathlineto{\pgfqpoint{2.313091in}{0.754180in}}%
\pgfpathlineto{\pgfqpoint{2.408205in}{0.746847in}}%
\pgfpathlineto{\pgfqpoint{2.519523in}{0.740565in}}%
\pgfpathlineto{\pgfqpoint{2.651977in}{0.735391in}}%
\pgfpathlineto{\pgfqpoint{2.816841in}{0.731275in}}%
\pgfpathlineto{\pgfqpoint{3.033136in}{0.728223in}}%
\pgfpathlineto{\pgfqpoint{3.346659in}{0.726204in}}%
\pgfpathlineto{\pgfqpoint{3.903955in}{0.725169in}}%
\pgfpathlineto{\pgfqpoint{4.323864in}{0.725000in}}%
\pgfpathlineto{\pgfqpoint{4.323864in}{0.725000in}}%
\pgfusepath{stroke}%
\end{pgfscope}%
\begin{pgfscope}%
\pgfsetrectcap%
\pgfsetmiterjoin%
\pgfsetlinewidth{0.803000pt}%
\definecolor{currentstroke}{rgb}{0.000000,0.000000,0.000000}%
\pgfsetstrokecolor{currentstroke}%
\pgfsetdash{}{0pt}%
\pgfpathmoveto{\pgfqpoint{0.625000in}{0.550000in}}%
\pgfpathlineto{\pgfqpoint{0.625000in}{4.400000in}}%
\pgfusepath{stroke}%
\end{pgfscope}%
\begin{pgfscope}%
\pgfsetrectcap%
\pgfsetmiterjoin%
\pgfsetlinewidth{0.803000pt}%
\definecolor{currentstroke}{rgb}{0.000000,0.000000,0.000000}%
\pgfsetstrokecolor{currentstroke}%
\pgfsetdash{}{0pt}%
\pgfpathmoveto{\pgfqpoint{4.500000in}{0.550000in}}%
\pgfpathlineto{\pgfqpoint{4.500000in}{4.400000in}}%
\pgfusepath{stroke}%
\end{pgfscope}%
\begin{pgfscope}%
\pgfsetrectcap%
\pgfsetmiterjoin%
\pgfsetlinewidth{0.803000pt}%
\definecolor{currentstroke}{rgb}{0.000000,0.000000,0.000000}%
\pgfsetstrokecolor{currentstroke}%
\pgfsetdash{}{0pt}%
\pgfpathmoveto{\pgfqpoint{0.625000in}{0.550000in}}%
\pgfpathlineto{\pgfqpoint{4.500000in}{0.550000in}}%
\pgfusepath{stroke}%
\end{pgfscope}%
\begin{pgfscope}%
\pgfsetrectcap%
\pgfsetmiterjoin%
\pgfsetlinewidth{0.803000pt}%
\definecolor{currentstroke}{rgb}{0.000000,0.000000,0.000000}%
\pgfsetstrokecolor{currentstroke}%
\pgfsetdash{}{0pt}%
\pgfpathmoveto{\pgfqpoint{0.625000in}{4.400000in}}%
\pgfpathlineto{\pgfqpoint{4.500000in}{4.400000in}}%
\pgfusepath{stroke}%
\end{pgfscope}%
\end{pgfpicture}%
\makeatother%
\endgroup%